\documentclass[a4paper]{article}

\usepackage{charter} 
\usepackage[utf8]{inputenc} 
\usepackage{microtype} 
\usepackage[italian,english]{babel}
\usepackage{amsmath,amsfonts,amssymb,amscd,lscape,amstext,amsthm}
\numberwithin{equation}{section}
\usepackage{float} 
\usepackage[final, colorlinks = true, 
            linkcolor = black, 
            citecolor = black]{hyperref} 
\usepackage{graphicx, multicol} 
\usepackage{color} 
\usepackage{marvosym, wasysym} 
\usepackage{rotating} 
\usepackage{censor} 
\usepackage{stmaryrd}
\usepackage{enumerate}
\usepackage{booktabs} 
\usepackage{tikz-qtree} 
\tikzset{every tree node/.style={align=center,anchor=north},
         level distance=2cm} 
\usepackage{fancyvrb}
\usepackage{verbatim}
\usepackage{subfig}
\usepackage{caption}
\usepackage{tabularx}

\usepackage{tabularx}
\usepackage{amsmath,amsfonts,amssymb,amscd,lscape,amstext,amsthm}

\newtheorem{lemma}{Lemma}[section]
\newtheorem{proposition}[lemma]{Proposition}
\newtheorem{theorem}[lemma]{Theorem}
\newtheorem{definition}{Definition}[section]
\newtheorem{remark}[lemma]{Remark}
\newtheorem{claim}[lemma]{Claim}

\newcommand {\R} {\mathbb{R}} 
 \newcommand {\N} {\mathbb{N}}

\newcommand{\diff}{\hspace{0.3em}\mathrm{d}}
\newcommand {\p} {\partial}

\newcommand{\de}{\mathcal{G}}
\newcommand{\di}{\Omega_D}
\newcommand{\F}{\mathcal{F}}
\newcommand{\G}{\mathcal{G}}

\begin{document}


\title{\Large\bf{Stable determination of an anisotropic inclusion in the Schr\"odinger equation from local Cauchy data}}
\author{Sonia Foschiatti \thanks{Dipartimento di Matematica e Geoscienze, Universit\`a degli Studi di Trieste, Italy, sonia.foschiatti@phd.units.it}
\qquad Eva Sincich\thanks{Dipartimento di Matematica e Geoscienze, Universit\`a degli Studi di Trieste, Italy, esincich@units.it}}
\date{}
\maketitle
		
\begin{abstract}
\noindent We consider the inverse problem of determining an inclusion contained in a body for a Schr\"odinger type equation by means of local Cauchy data. Both the body and the inclusion are made by inhomogeneous and anisotropic materials. Under mild a priori assumptions on the unknown inclusion, we establish a logarithmic stability estimate in terms of the local Cauchy data. In view of possible applications, we also provide a stability estimate in terms of an ad-hoc misfit functional.
\end{abstract}
		


\section{Introduction}\label{sec1}

\setcounter{equation}{0} 
\noindent The paper adresses the inverse boundary value problem of determining an inclusion $D$ contained in a body $\Omega$ for a Schr\"odinger type equation by means of complete measurements on a portion $\Sigma$ of the boundary. More precisely, we assume that $\Omega$ is a bounded domain in $\R^n$, $n\geq 3$ and $D$ is an open set contained in $\Omega$. Let both the body $\Omega$ and the inclusion $D$ be made by different inhomogeneous and anisotropic materials. For any $f\in H^{\frac{1}{2}}_{00}(\Sigma)$, consider the weak solution $u\in H^1(\Omega)$ to the Dirichlet problem

\begin{equation}\label{main}
\begin{cases}
\mbox{div}(\sigma \nabla u) + q u = 0, &\text{in }\Omega,\\
u = f &\text{on }\p \Omega.
\end{cases}
\end{equation}
with
\begin{equation}\label{conductivity}
\sigma(x)=\left(a_b(x)+(a_D(x)-a_b(x))\chi_D(x)\right)A(x)
\end{equation}
and 
\begin{equation}\label{potential}
q(x)=q_b(x)+(q_D(x)-q_b(x))\chi_D,
\end{equation}
where $a_b, q_b$ and $a_D, q_D$ are the scalar parameters of the background body $\Omega$ and the inclusion $D$, respectively, $\chi_D$ is the characteristic function of $D$ and $A(x)$ is a matrix-valued function. We denote by $\mathcal{C}^{\Sigma}_D$ the set of all the possible Cauchy data $(u\vert_{\Sigma}, \sigma\nabla u\cdot\nu\vert_{\Sigma})$ associated to the problem, where $\nu$ is the outer unit normal of $\Omega$ at $\Sigma$.

\noindent The inverse problem consists in the determination of $D$ given $\mathcal{C}^{\Sigma}_D$.

\noindent For $q=0$, namely when one deals with the conductivity equation, the direct problem is well-posed and one can define the so-called Dirichlet-to-Neumann map 
\begin{align}
\Lambda_D:H^{\frac{1}{2}}_{00}(\Sigma) &\rightarrow H^{-\frac{1}{2}}_{00}(\Sigma)\notag\\
f &\rightarrow \sigma \nabla u\cdot \nu\vert_{\Sigma},
\end{align}
which, roughly speaking, assigns to each electric potential prescribed on $\Sigma$ the corresponding current density. 
When no sign nor spectrum condition on $q$ are assumed, the existence of such operator is not guaranteed.

\noindent Notice that the inverse problem we are discussing encompasses the classical inverse conducting problem for a special class of anisotropic conductivities $\sigma$ of the form \eqref{conductivity} as well as the reduced wave equation $\triangle u + k^2 c^{-2} u=0$.

\noindent The prototype of this class of inverse problems is the determination of an inclusion in an isotropic electrostatic conductor by means of full boundary measurements of the electric potential and the current flux. The uniqueness issue for such an inverse problem was solved by Isakov in  \cite{v.isakov1988} by combining the Runge approximation theorem with the use of solutions with Green's function type singularities. The stability counterpart of Isakov's result has been tackled by Alessandrini and Di Cristo in \cite{g.alessandrinim.dicristo2005} by providing a logarithmic stability estimate. Their argument is still based on singular solution method, whereas Runge approximation argument has been replaced by the quantitative unique continuation estimates, since they seem to be more suitable for stability purposes. This strategy has inspired a line of research in which some methods and results have been extended to more complicated equations and systems (see for instance \cite{m.dicristo2009}, \cite{m.dicristoy.ren2017}, \cite{g.alessandrinim.dicristoa.morassie.rosset2014}, \cite{a.morassie.rosset2016}). More recently, in the spirit of \cite{g.alessandrinis.vessella2005}, where it has been observed that an improvement of the stability rate is possible for finite dimensional unknown conductivities, Lipschitz stability estimates for poligonal or polyhedral inclusions have been provided mainly in the context of the conductivity equation and the Helmholtz equation (see \cite{e.berettae.francini2020}, \cite{e.berettae.francinis.vessella2021}, \cite{e.berettam.v.dehoope.francinis.vessella2015}, \cite{a.asprie.berettae.francinis.vessella2022}). Another class of problems which is a particular instance of our problem is the optical tomography, which is mostly studied in medical imaging to infer the properties of a tissue (see for instance \cite{s.rarridgew1999}).
\bigskip

\noindent The purpose of our work is to prove the continuous dependence of $D$ in the Hausdorff metric from the local Cauchy data via a modulus of continuity of logarithmic type. Altough our strategy has been stimulated by the one introduced in \cite{g.alessandrinim.dicristo2005}, the more general context we are dealing with requires the development of new arguments and tools.

\noindent Let us summarize the main steps of our proof with the related issues. 

 \begin{enumerate}[i)]
  \item We introduce the space of local Cauchy data $\mathcal{C}_D^{\Sigma}$ measured on the accessible portion $\Sigma$ and their metric structure as a subspace of a Hilbert space. As in \cite{g.alessandrinim.v.dehoopr.gaburroe.sincich2018}, we express the error on the boundary data in terms of the so-called distance (or aperture) between spaces of Cauchy data.  We also wish to recall that when the local Dirichlet to Neumann map $\Lambda^{\Sigma}_D$ exists, for instance when $q=0$, the local Cauchy data are the graph of $\Lambda^{\Sigma}_D$ and the distance between two sets of Cauchy data is equivalent to the norm of the difference of the corresponding local Dirichlet to Neumann maps. 
 \noindent  For two inclusions $D_1$ and $D_2$, we consider the corresponding local Cauchy data sets $\mathcal{C}_{D_1}^{\Sigma},\mathcal{C}_{D_2}^{\Sigma}$ collected on $\Sigma$, the distance $d(\mathcal{C}_{D_1}^{\Sigma},\mathcal{C}_{D_2}^{\Sigma})$ and, by using the Alessandrini identity argument, we have established the following inequality 
 
 \begin{eqnarray}\label{1}
\left| \int_{\Omega} (\sigma_2-\sigma_1)\nabla u_1\cdot \nabla u_2 + \int_{\Omega} (q_1-q_2) u_1 u_2 \right|\,\leq d(\mathcal{C}_{D_1}^{\Sigma},\mathcal{C}_{D_2}^{\Sigma})\,\, \|(u_1,\sigma_1\nabla u_1\cdot\nu)\|_{\mathcal{H}}\,\, \|(\bar{u}_2,\sigma_2\nabla \bar{u}_2\cdot \nu)\|_{\mathcal{H}},
 \end{eqnarray}
(see Subsection 2.3). Here, $u_1$ and $u_2$ are solution to \eqref{main} when $D=D_1=D_2$ respectively. 
\\

 \item Since the equation at hand might be in the eigenvalue regime, we need to construct and to estimate Green's function with mixed type boundary conditions, namely Dirichlet type on a portion of the boundary and complex-valued Robin type in the remaining one. Such a boundary value problem with local complex Robin condition is well-posed as proved by Bamberger and Duong in \cite{a.bambergert.ha.duong1987}.  We adapt the argument introduced in \cite{g.alessandrinim.v.dehoopr.gaburroe.sincich2018} to our equation, in which the principal part has a matrix-valued leading coefficient that might have a discontinuity across the boundary of the inclusion $D$. We overcome the leading term discontinuity issue by using a quite recent result of propagation of smallness due to Carstea and Wang \cite{c.i.carsteaj.n.wang2020} for a scalar second order elliptic equation in divergence form whose leading coefficients are Lipschitz continuous on two sides of a $C^2$ hypersurface that crosses the domain, but may have jumps across this hypersurface.
 We consider the above inequality for singular solutions $u_1(\cdot)= G_1(\cdot,y)$ and $u_2(\cdot)=G_2(\cdot,w)$ defined on a larger domain. 
  \noindent Focusing on the right hand side of $\eqref{1}$, we introduce the function
\begin{eqnarray}
f(y,w) = S_{D_1}(y,w)-S_{D_2}(y,w),
\end{eqnarray}
which is a solution of our underlying equation in the connected component $\mathcal{G}$ of $\mathbb{R}^n \setminus (\overline{D_1\cup D_2})$ which contains $\mathbb{R}^n\setminus \overline{\Omega}$. Moreover, in the case in which $y,z$ are placed outside $\Omega$, $f$ is controlled in terms of $d(\mathcal{C}_{D_1}^{\Sigma},\mathcal{C}_{D_2}^{\Sigma})$. We propagate the smallness of $f$ as $y,w$ move inside $\Omega$ within $\mathcal{G}$. We wish to underline a delicate point of the proof which is the fact that we can only perform unique continuation estimates near point in a subset, say $V$, of the boundary of $D_1\cup D_2$, that can be reached from $\mathcal{G}$ in a quantitative form. This involves the use of chain of balls whose numbers are suitably bounded and whose radii must be bounded from below (see \cite{g.alessandrinie.sincich2013}, \cite{g.alessandrinim.dicristoa.morassie.rosset2014} and also  \cite{l.rondie.sincichm.sini2021} for a related argument).  In this respect, a crucial step is that under the a priori regularity assumptions on $D_1,D_2$, we can prove that there exists a point $P\in\partial D_1\cap V$ such that the Hausdorff distance between $\partial D_1$ and $\partial D_2$ is dominated by the distance $dist(P, D_2)$ .

\item We show that when $y=w$ tends to a point $P$ of $\partial D_1\setminus \overline{D_2}$, $f(y,y)$ blows up. The combination of such a blow-up and the control of $f(y,y)$ in terms of $d(\mathcal{C}_{D_1}^{\Sigma},\mathcal{C}_{D_2}^{\Sigma})$ discussed above leads to the logarithmic estimates 
\begin{eqnarray}
d_H(\partial D_1, \partial D_2)\le C \left| \log (d(\mathcal{C}_{D_1}^{\Sigma},\mathcal{C}_{D_2}^{\Sigma}))\right|^{-\eta} \ .
\end{eqnarray}
 The new features of anisotropic and inhomogeneous leading coefficient, as well as the additional zero-order term, require a careful analysis of the asymptotic behaviour of the singular solutions $G_i(\cdot, y)$ when the pole $y$ approaches to the inclusion $D_i, \ i=1,2$ and serves as a tool to achieve the blow up estimate of $f(y,y)$.

 \end{enumerate}

\noindent We believe that the present study might be a theoretical building block for future Lipschitz stability result and corresponding numerical reconstruction procedure under the a priori assumption of poygonal or polyhedral inclusion. For this reason, we also choose to provide the following stability estimate
 \begin{eqnarray}\label{misfitmain}
 d_H(\partial D_1, \partial D_2)\le C  \left| \log  \mathcal{J}(D_1,D_2)\right|^{-\eta},
 \end{eqnarray}
in the present context of a general inclusion with $C^2$ boundary and in terms of a misfit functional 
 \begin{multline}\label{misfitma}
 \mathcal {J}(D_1,D_2) = \int_{D_y\times D_z} \Big|\int_{\Sigma} \Big[ \sigma_{1}(x)\nabla G_1(x,y)\cdot\nu(x)\: G_2(x,z) - \\
- \sigma_{2}(x)\nabla G_2(x,z)\cdot\nu(x)\: G_1(x,y)\Big]\:dS(x)\Big|^2dy\, dz,
 \end{multline}
where the $D_y, D_z$ are suitably chosen sets compactly contained in $\R^n\setminus \bar{\Omega}$ (see Section \ref{5} for a more precise definition). We expect that for polygonal or polyhedra inclusions, the logarithmic rate in \eqref{misfitmain} might be improved up to a H\"older type stability. As shown in \cite{g.alessandrinim.v.dehoopf.faucherr.gaburroe.sincich2019} (see also \cite{s.foschiattie.sincichr.gaburro2021}), the above mentioned H\"older estimate may be suitable to the numerical reconstructions and to the use of the singular solutions method. The use of a misfit functional of such kind suggests that the knowledge of the full Dirichlet to Neumann map or the full local Cauchy data set are not necessary, and that it suffices to sample them on Green’s type
functions with sources placed outside the physical domain.
 
\noindent The paper is organised as follows. In Section \ref{2} we introduce the a priori assumptions, we define the local Cauchy data and state the main theorem. In Section \ref{3} we introduce the geometric lemmas and we prove the main theorem.  In Section \ref{4} we introduce and prove technical propositions. In particular, we construct the Green function (Lemma \ref{lemma2}) and we prove the upper bound for $f$ (Proposition \ref{prop-bis}) and the lower bound for $f$ (Proposition \ref{prop5}). In Section \ref{5} we derive a stability result in terms of the misfit functional \eqref{misfitma}.


\section{Main Result}\label{2}
		
\setcounter{equation}{0}

\subsection{Notation and definitions}
	
\noindent Denote a point $x\in \mathbb{R}^n$ by $x=(x',x_n)$, where $x'\in\mathbb{R}^{n-1}$ and $x_n\in\mathbb{R}$, $n\geq 3$. Denote with $B_r(x)\subset \mathbb{R}^{n}$, $B_r'(x')\subset \mathbb{R}^{n-1}$ the open balls centred at $x$, $x'$ respectively with radius $r$, with $Q_r(x)$ the cylinder
\[Q_r(x)=B_r'(x')\times(x_n-r,x_n+r).\]
	
\noindent Set $B_r=B_r(O)$, $Q_r=Q_r(O)$, the positive real half space $\R^n_+ = \{ (x',x_n)\in \R^n\,:\, x_n>0 \}$, the negative real half space $\R^n_- = \{ (x',x_n)\in \R^n\,:\, x_n<0 \}$, the positive semisphere centred at the origin $B^+_r = B_r\cap \R^n_+$, the negative semisphere centred at the origin $B^-_r = B_r\cap \R^n_-$.
	
\begin{definition}[$C^{2}$ regularity]\label{firstdefin}
Let $\Omega\subset \R^n$ be a bounded domain. We say that a portion $\Sigma$ of $\partial\Omega$ is of $C^{2}$ class with constants $r_0,L>0$, if, for any point $P\in\Sigma$, there exists a rigid transformation of coordinates under which $P=O$ and
\begin{align*}
\Omega\cap Q_{r_0}= \left\{ x\in Q_{r_0}\,:\,x_n>\varphi(x') \right\},
\end{align*}
		
\noindent where $\varphi\in C^{2}(B'_{r_0})$ is such that 
\begin{align*}
\varphi(O)=|\nabla\varphi(O)|=0 \hspace{1.5cm}\text{and}\hspace{1.5cm}\|\varphi\|_{C^{2}(B'_{r_0})}\leq Lr_0.
\end{align*}
\end{definition}

\begin{remark}
The $C^{2}$-norm in Definition \ref{firstdefin} is normalized so that its terms are dimensionally homogeneous, i.e.
\[
\|\varphi\|_{C^{2}(B'_{r_0})} = \sum_{i=0}^2 r_0^i \|\nabla^i \varphi\|_{L^{\infty}(B'_{r_0})}, 
\]
\end{remark}

\begin{definition}
Let be $\Omega$ a domain of $\R^n$. We say that a portion $\Sigma$ of the boundary of $\Omega$ is a flat portion of size $r_0$ if there exist a point $P\in \Sigma$ and a rigid transformation of coordinates under which $P=O$ and 
\begin{align*}
\Sigma\cap Q_{\frac{r_0}{3}} &= \left\{ x\in Q_{\frac{r_0}{3}}\,:\, x_n=0\right\},\\
\Omega\cap Q_{\frac{r_0}{3}} &= \left\{ x\in Q_{\frac{r_0}{3}}\,:\, x_n>0\right\},\\
(\R^n\cap\Omega)\cap Q_{\frac{r_0}{3}} &= \left\{ x\in Q_{\frac{r_0}{3}}\,:\, x_n<0\right\}.
\end{align*}
\end{definition}

\begin{definition}
The \textit{Hausdorff distance} between two bounded closed subsets $D_1$ and $D_2$ in $\R^n$ is defined as
\[
d_H( D_1, D_2) := \max \Big\{ \sup_{w\in D_2} dist(w, D_1), \sup_{w\in D_1} dist(w, D_2)\Big\}.
\]
\end{definition}

\subsection{A priori information}\label{apriorinfo}
In this section, we introduce the a priori information on the domain $\Omega$, the inclusion $D$ and the coefficients $\sigma$ and $q$.

\begin{enumerate}[i)]
\item \textit{Domain.}
The set $\Omega$ is a bounded domain in $\R^n$ such that
\begin{equation}\label{omegabegin}
\p\Omega \text{ is of Lipschitz class with constants }r_0, L,
\end{equation}
\begin{equation}\label{omegaend}
|\Omega|\leq Nr_0^n,
\end{equation}
where $r_0, L, N$ are given positive constants, $n\geq 3$.
\item \textit{Open portion of the boundary}.
\begin{equation}
\Sigma\subset \p\Omega \text{ is a non-empty, flat open portion,}
\end{equation}
\item \textit{Inclusion.} Let $D$ be a connected subset of $\Omega$ such that 
\begin{equation}\label{inclusionbegin}
D\subset\subset\Omega\quad \text{ and }\quad \text{dist}(\p D,\p\Omega)\geq\delta_0>0,
\end{equation}
\begin{equation}
\p D \text{ is of }C^{2}\text{ class with constants }r_0, L, 
\end{equation}
\begin{equation}\label{inclusionend}
\Omega\setminus \bar{D} \text{ is connected.}
\end{equation}

\item \textit{Parameters.} 
The coefficient $\sigma\in L^{\infty}(\Omega,Sym_n)$ has the following structure
\begin{equation}\label{condbegin}
\sigma(x) = \left(a_b(x) + (a_D(x) - a_b(x))\chi_D(x)\right)A(x),
\end{equation}
where the scalar functions $a_b, a_D$ are in $C^{0,1}(\bar{\Omega})$. Moreover, there exist $\bar{\gamma}>1$, $\eta_0>0$ such that
\begin{equation}
\bar{\gamma}^{-1}\leq a_b(x), a_D(x) \leq \bar{\gamma}, \quad \mbox{for }x\in \Omega,
\end{equation}
\begin{equation}\label{lineareta}
(a_D(x)-a_b(x))^2\geq \eta_0^2>0, \quad \mbox{for }x\in \Omega.
\end{equation}
The real $n\times n$ matrix-valued function $A(x)$ is a symmetric Lipschitz continuous function such that there exists $\bar{A}>0$ for which
\begin{equation}
\|A\|_{C^{0,1}(\Omega)}\leq \bar{A}.
\end{equation}
The matrix-valued function $\sigma$ satisfies the uniform ellipticity condition, i.e. there exists a constant $\bar{\lambda}>1$ such that
\begin{equation}\label{condend}
\bar{\lambda}^{-1} |\xi|^2 \leq \sigma(x)\xi\cdot\xi \leq \bar{\lambda} |\xi|^2, \quad \mbox{for a.e. }x\in \Omega, \mbox{ for all }\xi\in \R^n.
\end{equation}
The scalar function $q$ has the form
\begin{equation}\label{zerorderbegin}
q(x) = q_b(x)+(q_D(x)-q_b(x))\chi_D(x), 
\end{equation}
where the functions $q_b, q_D$ are in $L^{\infty}(\bar{\Omega})$. Moreover, for $\bar{\gamma}>0$ and $\eta_0>0$, it holds
\begin{equation}\label{zerorderend}
\|q\|_{L^{\infty}(\Omega)}\leq \bar{\gamma}.
\end{equation}
\end{enumerate}
The set $\{n, N, r_0, L, \bar{A}, \bar{\gamma}, \bar{\lambda}, \delta_0, \eta_0\}$ is called the \textit{a priori data.}


\subsection{Local Cauchy data and the main result}

Consider a domain $\Omega$, an inclusion $D$ satisfying \eqref{omegabegin}-\eqref{omegaend} and \eqref{inclusionbegin}-\eqref{inclusionend} respectively. Let $\Sigma$ be the accessible portion of $\p\Omega$ where the boundary measurements are taken. A Cauchy data set is a collection of boundary data measurements associated with an inclusion. Before giving the formal definition of the notion of local Cauchy data, we introduce suitable trace spaces. Let
\[
H^{\frac{1}{2}}_{co}(\Sigma) = \left\{f\in H^{\frac{1}{2}}(\p\Omega)\,:\,\text{supp}(f)\subset \Sigma\right\}.
\]
be the trace space that contains all the trace functions which are compactly supported in $\Sigma$. Denote with $H^{\frac{1}{2}}_{00}(\Sigma)$ its closure under the norm $\|\cdot\|_{H^{\frac{1}{2}}(\p\Omega)}$. Similarly, let
\[
H^{\frac{1}{2}}_{co}(\p\Omega\setminus\bar{\Sigma}) = \left\{f\in H^{\frac{1}{2}}(\p\Omega)\,:\,\text{supp}(f)\subset \p\Omega\setminus\bar{\Sigma}\,\right\}
\]
Denote with $H^{\frac{1}{2}}_{00}(\p\Omega\setminus\bar{\Sigma})$ its closure under the norm $\|\cdot\|_{H^{\frac{1}{2}}(\p\Omega)}$. Let $H^{-\frac{1}{2}}(\p\Omega)$ be the dual space of $H^{\frac{1}{2}}_{co}(\p\Omega)$.

\begin{definition}
The Cauchy data set on $\Sigma$ associated with the inclusion $D$ is defined as the set
\begin{align}
\mathcal{C}_D^{\Sigma}(\Omega) = \Big\{ (f,g)\in H^{\frac{1}{2}}_{00}(\Sigma)\times H^{-\frac{1}{2}}(\p \Omega)\,:\,&\exists u\in H^1(\Omega) \text{ weak solution to }\notag\\
&\text{div}(\sigma\nabla u) + qu = 0 \qquad \text{in }\Omega,\notag\\
&u\vert_{\p\Omega} = f, \qquad \sigma\nabla u\cdot \nu\vert_{\p\Omega}=g \Big\}
\end{align}
\end{definition}

\noindent Let 
\[
H^{-\frac{1}{2}}_{00}(\p\Omega\setminus\bar{\Sigma})=\left\{ \psi\in H^{-\frac{1}{2}}(\p\Omega)\,:\, \langle \psi,\varphi\rangle =0, \quad \forall \varphi\in H^{\frac{1}{2}}_{00}(\Sigma)\,\right\}
\]
where $\langle \psi,\varphi\rangle$ represents the duality between the spaces $H^{-\frac{1}{2}}(\p\Omega)$ and $H^{\frac{1}{2}}(\p\Omega)$ based on the inner product on $L^2(\p\Omega)$
\[
\langle \psi,\varphi\rangle = \int_{\p\Omega} \psi(x)\cdot\overline{\varphi(x)}\diff x.
\]

\noindent Define $H^{\frac{1}{2}}(\p\Omega)\vert_{\Sigma}$ and $H^{-\frac{1}{2}}(\p\Omega)\vert_{\Sigma}$ as the restrictions to $\Sigma$ of the trace spaces $H^{\frac{1}{2}}(\p\Omega)$ and $H^{-\frac{1}{2}}(\p\Omega)$ respectively. These trace spaces can be equivalently defined as quotient spaces via the following relation:
\[
\varphi \sim \psi \qquad \iff \qquad \varphi-\psi \in H^{\frac{1}{2}}_{00}(\p\Omega\setminus\bar{\Sigma})
\]
so that
\[
H^{\frac{1}{2}}(\p\Omega)\vert_{\Sigma} = H^{\frac{1}{2}}(\p\Omega) \slash \sim\,\, = H^{\frac{1}{2}}(\p\Omega) \slash H^{\frac{1}{2}}_{00}(\p\Omega\setminus\bar{\Sigma}).
\]
Similarly,
\[
H^{-\frac{1}{2}}(\p\Omega)\vert_{\Sigma} = H^{-\frac{1}{2}}(\p\Omega)\slash H^{-\frac{1}{2}}_{00}(\p\Omega\setminus\bar{\Sigma}).
\]

\begin{definition}
The local Cauchy data on $\Sigma$  associated with the inclusion $D$ whose first component vanishes on $\p\Omega\setminus\Sigma$ is defined as
\begin{align}
\mathcal{C}_D^{\Sigma}(\Sigma)= \Big\{ (f,g)\in H^{\frac{1}{2}}_{00}(\Sigma)\times H^{-\frac{1}{2}}(\p\Omega)\vert_{\Sigma}\,:\,&\exists u\in H^1(\Omega) \text{ weak solution to }\notag\\
&\mbox{div}(\sigma\nabla u) + qu = 0 \qquad \text{in }\Omega,\notag\\
&u = f \qquad \text{on }\p\Omega,\notag\\
&\langle\sigma\nabla u\cdot \nu\vert_{\p\Omega},\varphi\rangle = \langle g, \varphi \rangle, \qquad \forall \varphi \in H^{\frac{1}{2}}_{00}(\Sigma) \Big\}.
\end{align}
\end{definition}
\noindent Notice that $\mathcal{C}^{\Sigma}_D(\Sigma)$ is a subspace of the product space $H^{\frac{1}{2}}_{00}(\Sigma)\times H^{-\frac{1}{2}}(\p\Omega)\vert_{\Sigma}$.

\noindent Set $\mathcal{H}:=H^{\frac{1}{2}}_{00}(\Sigma)\times H^{-\frac{1}{2}}(\p\Omega)\vert_{\Sigma}$. 
\noindent  Notice that $\mathcal{H}$ is a Hilbert space with norm
\[
\|(f,g)\|_{\mathcal{H}}= \Big( \|f\|^2_{H^{\frac{1}{2}}_{00}(\Sigma)} + \|g\|^2_{H^{-\frac{1}{2}}(\p\Omega)\vert_{\Sigma}}\Big)^{\frac{1}{2}},\qquad (f,g)\in\mathcal{H}.
\]
\vspace{0.5cm}

\noindent Let $D_1,D_2$ be two inclusions satisfying \eqref{inclusionbegin}-\eqref{inclusionend}. In order to simplify the notation, set $\mathcal{C}_i=\mathcal{C}^{\Sigma}_{D_i}(\Sigma)$ for $i=1,2$ which denote the local Cauchy data associated with the corresponding inclusions. 

\noindent As in \cite{g.alessandrinim.v.dehoopf.faucherr.gaburroe.sincich2019}, we find convenient to introduce the notion of \textit{distance} (or \textit{aperture}) between closed subspaces $\F$ and $\G$ of a Hilbert space $\mathcal{H}$ as the quantity
\[
d(\F,\G)=\max\Big\{ \sup_{h\in\G, h\neq 0} \inf_{k\in \F} \frac{\|h-k\|}{\|h\|}, \sup_{k\in\F, k\neq 0} \inf_{h\in \G} \frac{\|h-k\|}{\|k\|} \Big\}.
\]

\noindent As a reference for the related theory and applications see \cite{t.kato1980}, \cite{a.knyazeva.jujunashvilim.argentati2010}. If $d(\F,\G)<1$, it is known that the two quantity in the maximum coincides (see \cite[Corollary 2.13]{a.knyazeva.jujunashvilim.argentati2010}), therefore we can assume that 
\begin{equation}\label{easydistance}
d(\F,\G) = \sup_{h\in\G, h\neq 0} \inf_{k\in \F} \frac{\|h-k\|}{\|h\|}.
\end{equation}
Choose $\F=\mathcal{C}_1$ and $\G=\mathcal{C}_2$, since we consider the case in which $d(\mathcal{C}_1,\mathcal{C}_2)<1$, then by \eqref{easydistance} the distance between two local Cauchy data has the form
\begin{equation}\label{cauchy}
d(\mathcal{C}_1,\mathcal{C}_2)= \sup_{(f_1,g_1)\in \mathcal{C}_1\setminus\{(0,0)\}} \inf_{(f_2,g_2)\in \mathcal{C}_2} \frac{\|(f_1,g_1)-(f_2,g_2)\|_{\mathcal{H}}}{\|(f_1,g_1)\|_{\mathcal{H}}}.
\end{equation}
Let $\omega:[0,+\infty)\rightarrow[0,+\infty)$ be a non decreasing, positive function such that
\begin{equation}\label{omega}
\omega(t)\leq |\ln t|^{-\eta}\hspace{0.5cm} \text{ for } t\in (0,1),\hspace{1.5cm} \omega(t)\rightarrow 0\hspace{0.5cm} \text{as}\hspace{0.5cm}t\rightarrow 0^+,
\end{equation}
where $\eta$ is a positive constant.

\begin{theorem}\label{theorem}
Let $\Omega\subset\R^n$, $n\geq 3$ be a bounded domain satisfying \eqref{omegabegin}-\eqref{omegaend} and let $D_1$, $D_2$ be two inclusions of $\Omega$ satisfying \eqref{inclusionbegin}-\eqref{inclusionend}. Let $\sigma_{1}$, $\sigma_{2}$ be the anisotropic conductivities satisfying \eqref{condbegin}-\eqref{condend} and let $q_1$, $q_2$ be the coefficients of the zero order term satisfying \eqref{zerorderbegin}-\eqref{zerorderend}.  Let $\mathcal{C}_1$, $\mathcal{C}_2$ be the local Cauchy data corresponding to the inclusions $D_1$, $D_2$, respectively. For $\epsilon\in(0,1)$, if $d(\mathcal{C}_1,\mathcal{C}_2)<\epsilon$, then 
\begin{equation}\label{stability}
d_H(\p D_1,\p D_2)\leq C \omega(\epsilon),
\end{equation}
where $C$ is a positive constant depending on the a priori data only and $\omega$ is defined in \eqref{omega}.
\end{theorem}

\section{Proof of Theorem \ref{theorem}}\label{3}
\setcounter{equation}{0}
To begin with, in the spirit of \cite{g.alessandrinim.dicristo2005}, we introduce the socalled \textit{modified distance}. 
 After that, we introduce the singular solutions and state Proposition \ref{prop-bis} and Proposition \ref{prop5} which allow us to upper-bound the function $f$ defined in \eqref{f} in terms of the distance between two Cauchy data sets and to lower bound $f$ in terms of the geometric quantities related with our problem. In the last part, we prove Theorem \ref{theorem}.

\subsection{Metric Lemmas}\label{preliminaryprop}
\noindent Let $\Omega\subset\R^n$ be a bounded domain satisfying \eqref{omegabegin}-\eqref{omegaend} and let $D_1$, $D_2$ be two inclusions contained in $\Omega$ satisfying \eqref{inclusionbegin}-\eqref{inclusionend}. Let $\sigma_{1}$ and $\sigma_{2}$ be the anisotropic conductivities satisfying \eqref{condbegin}-\eqref{condend}, and $q_1$, $q_2$ satisfying \eqref{zerorderbegin}-\eqref{zerorderend}

\noindent Denote with $\G$ the connected component of $\Omega\setminus (D_1\cup D_2)$ whose boundary contains $\p \Omega$. Set $\di=\Omega\setminus \G$. 

\begin{definition}
The modified distance between two closed subsets $D_1$ and $D_2$ of $\R^n$ is the number
\begin{equation}
d_{\mu}(D_1,D_2) = \max \Big\{ \sup_{x\in \p D_1\cap\p\Omega_D} \mbox{dist}(x,D_2), \sup_{x\in \p D_2\cap\p\Omega_D} \mbox{dist}(x,D_1)\Big\}.
\end{equation}
\end{definition}
\noindent The introduction of another map that quantifies the distance between two inclusions $D_1, D_2$ is justified by the fact that the point in which the Hausdorff distance is attained may lay on $\p D_1\cup \p D_2$ but not on $\p \Omega_D$. This is an obstruction in view of the application of the propagation of smallness argument, since in order to reach that point from $\p\Omega$ we have to cross $\p D_1\cup \p D_2$.

\noindent In general, the modified distance is not a distance and it does not bound the Hausdorff distance from above (see \cite{g.alessandrinie.berettae.rossets.vessella2000} for a counterexample). Under our assumptions on the inclusions, the following lemma guarantees that in our case $d_{\mu}$ dominates $d_{\mathcal{H}}$ (for the proof see \cite[Proposition 3.3]{g.alessandrinim.dicristo2005}).
\begin{lemma}\label{lemma1}
Let $\Omega$, $D_1$, $D_2$ be, respectively, a bounded domain satisfying \eqref{omegabegin}-\eqref{omegaend} and two inclusions satisfying \eqref{inclusionbegin}-\eqref{inclusionend}. Then there is a positive constant $c$ which depends only on the a priori data such that
\begin{equation}
d_{\mathcal{H}}(\p D_1, \p D_2)\leq c d_{\mu}(D_1, D_2).
\end{equation}
\end{lemma}

\noindent Let us remark now that for simplicity we assume that there exists a point on $\p D_1\cap \p\Omega_D$ which realizes the modified distance. We denote that quantity simply as $d_{\mu}$.
\bigskip

\noindent In order to apply the quantitative estimates for unique continuation, based on the iterated application of the three-spheres inequality, it is important to control the radii of these balls from below and to avoid the case in which points of $\partial\Omega_D$ are not reachable by such a chain of balls. Hence, we find convenient to introduce here ideas first presented in \cite{g.alessandrinie.sincich2013} and then applied by \cite{g.alessandrinim.dicristoa.morassie.rosset2014} in the elasticity case.

\noindent Let $O$ denote the origin in $\R^n$, $\nu$ be a unit vector, $h>0$ and $\theta\in\left(0,\frac{\pi}{2}\right)$ be an angle. We recall that the closed truncated cone with vertex at $O$, axis along the direction $\nu$, height $h$ and aperture $2\theta$ is given by
\begin{equation}
C(O,\nu,h,\theta)=\left\{x\in\R^n:\, |x-(x\cdot\nu)\nu|\leq |x|\sin\theta,\, 0\leq x\cdot \nu\leq h\right\}.
\end{equation}
In our case, for $d, R>0$ such that $R<d$, fixed a point $Q=-de_n$, we consider the closed cone 
\[
C\left(O,-e_n,\frac{d^2-R^2}{d},\arcsin \frac{R}{d}\right)
\]
whose oblique sides are tangent to the sphere $\partial B_R(O)$. 

\noindent Let us pick a point $P\in \partial D_1\cap \partial\Omega_D$ and let $\nu$ be the outer unit normal to $\partial D_1$ at $P$. For a suitable $d>0$, let $[P+d\nu,P]$ be the segment contained in $\R^n\setminus\bar{\Omega}_D$. For $P_0\in \R^n\setminus\bar{\Omega}$, let $\gamma:[0,1]\rightarrow \R^n\setminus\bar{\Omega}_D$ be the path contained in $(\R^n\setminus\bar{\Omega}_D)_d$ such that $\gamma(0)=P_0$ and $\gamma(1)=P+d\nu$. Consider the following neighbourhood of $\gamma\cup [P+d\nu,P]\setminus\{P\}$ formed by a tubular neighbourhood of $\gamma$ attached to a cone with vertex at $P$ and axis along $\nu$,
\begin{equation}\label{tubular}
V(\gamma) = \left[\bigcup_{S\in\gamma} B_R(S)\right]\cup C\left(P,\nu,\frac{d^2-R^2}{d},\arcsin \frac{R}{d}\right).
\end{equation}
Notice that $V(\gamma)$ will depend on the parameters $d, R$, that will be chosen step by step. The following result  guarantees the application of the three-sphere inequality along the tubular neighbourhood contained in $\R^n\setminus\bar{\Omega}_D$ centred at any point of $\partial D_1\cap \partial \Omega_D$ (see \cite[section 11]{g.alessandrinim.dicristoa.morassie.rosset2014} for a proof).

\begin{lemma}\label{metric-lemma}
There exists positive constants $\bar{d}, c_1$ where $\frac{\bar{d}}{r_0}, c_1$ depend on $L$ and there exists a point $P\in \partial D_1$ satisfying 
\begin{equation}\label{uppermodist}
c_1 d_{\mu} \leq dist(P,D_2),
\end{equation}
and such that, for any $P_0\in B_{\frac{r_0}{16}}(P_0)$, where $B_{\frac{r_0}{16}}(P_0)\subset\subset \R^n\setminus \bar{\Omega}_D$, there exists a path $\gamma \subset\R^n\setminus\bar{\Omega}_D$ joining $P_0$ to $P+\bar{d}\nu$ where $\nu$ is the outer unit normal to $D_1$ at $P$ such that, if we choose the coordinate system in which $P$ coincides with the origin and $\nu=-e_n$, then
\begin{equation}
V(\gamma)\subset \R^n\setminus \Omega_D,
\end{equation}
where $V(\gamma)$ is the tubular neighbourhood introduced in \eqref{tubular}, provided that $R=\frac{\bar{d}}{\sqrt{1+L^2_0}}$, where $L_0>0$ depends only on $L$.
\end{lemma}

\vspace{0.5cm}

\subsection{Proof of Theorem \ref{theorem}}

\noindent Before proving the theorem, fix a point on the surface $\Sigma$ so that, up to a rigid transformation, it coincides with the origin. We define $D_0$ as the bounded domain with Lipschitz constants $r_0>0$, $L>0$ of the form
\begin{equation}\label{dtilde}
D_0=\left\{x\in(\mathbb{R}^n\setminus\overline{\Omega})\cap B_{r_0}\: :\:|x_i|<r_0,\:i=1,\dots , n-1,\:\:-r_0<x_n<0\right\},
\end{equation}
such that $\p D_0\cap\p\Omega\,\,\subset\subset\, \Sigma$. We define the augmented domain $\Omega_0$ as the set 
\begin{equation}\label{aumenteddomain}
\Omega_0\,\,=\,\,\overset{\circ}{\overline{\Omega\cup D_0}}.
\end{equation}
\noindent It turns out that $\Omega_0$ is of Lipschitz class with constants $r_0$ and $\tilde{L}$, where $\tilde{L}$ depends on $L$ only. Moreover, let $\Sigma_0\subset \p D_0$ be a nonempty portion of the boundary of $D_0$ of the form
\[
\Sigma_0 = \left\{ x\in \Omega_0\,:\, |x_i|\leq r_0, \:\: x_n=-r_0\right\}
\]
such that $\Sigma_0\cap\p\Omega=\emptyset$.

\noindent Consider $D_1$, $D_2$ the two inclusions of $\Omega$ satisfying \eqref{inclusionbegin}-\eqref{inclusionend}. Without loss of generality, we can extend the corresponding coefficients $\sigma_{1}, \sigma_{2}, q_{1}, q_{2}$ to the augmented domain $\Omega_0$ by setting their value equal to the identity matrix on $D_0$, so that they are of the form
\begin{displaymath}
\begin{array}{ll}
\sigma_{i}(x) = \Big(a_b(x) + (a_{D_i}(x) - a_b(x))\chi_{D_i}(x)\Big)A(x), &\textrm{$\textnormal{for any }x\in\Omega$},\\
\sigma_{i}\vert_{D_0} = I, &q_{i}\vert_{D_0}=1.
\end{array}
\end{displaymath}
We denote with the same symbol the extended coefficients $\sigma_{i}$ and $q_{i}$ for $i=1,2$.

\noindent By now, let us drop the pedix $i$ and consider the coefficients $\sigma$, $q$ associated with a generic inclusion $D$. Denote with $G$ the Green function associated with the operator $\mbox{div}(\sigma(\cdot)\nabla\cdot)+q(\cdot)$ on the augmented domain $\Omega_0$. For every $y\in D_0$, $G(\cdot,y)$ is the distributional solution to the Dirichlet problem
\begin{equation}\label{greensystem}
\left\{ \begin{array}{ll}
\mbox{div}(\sigma(\cdot)\nabla G(\cdot,y)) + q(\cdot) G(\cdot,y) = - \delta(\cdot-y) &\textrm{$\textnormal{in}\quad \Omega_0$},\\
G(\cdot,y)=0 &\textrm{$\textnormal{on}\quad\p\Omega_0\setminus \Sigma_0$},\\
\sigma(\cdot) \nabla G(\cdot,y)\cdot \nu(\cdot) + i G(\cdot,y)=0 &\textrm{$\textnormal{on}\quad\Sigma_0$},
\end{array}
\right.
\end{equation}
where $\delta(\cdot - y)$ is the Dirac distribution centred at $y$ and $\nu$ is the outer unit normal at $\Sigma_0$.
The following property holds true for the Green's functions (see Lemma \ref{lemma2}):
\begin{equation}\label{boundgreen}
0<|G(x,y)|<C |x-y|^{2-n},\qquad \forall x\neq y,
\end{equation}
where $C$ is a positive constant depending on $\lambda$ and $n$.
\bigskip

\noindent Let us now consider $G_1$ and $G_2$ the Green functions solutions to \eqref{greensystem} associated to the inclusions $D_1$ and $D_2$ respectively. Fix $i$ in \eqref{greensystem}, then multiply the first equation by $G_j(\cdot,y)$ for $j\neq i$ and integrate by parts on $\Omega$. Repeat the same procedure interchanging the role of $i$ and $j$. Subtracting the two quantities obtained leads to the following equality
\begin{align}\label{green}
\int_{\Sigma} &\left[ \sigma_{1}(x)\nabla G_1(x,y)\cdot\nu(x)\: G_2(x,z) - \sigma_{2}(x)\nabla G_2(x,z)\cdot\nu(x)\: G_1(x,y)\right] \:\diff S(x) =\notag \\
&= \int_{\Omega} ( \sigma_{1}(x) - \sigma_{2}(x)) \nabla G_1(x,y)\cdot \nabla G_2(x,z) \diff x + \int_{\Omega} ( q_{2}(x) - q_{1}(x)) G_1(x,y)\: G_2(x,z) \diff x. 
\end{align}
Define
\begin{align}\label{S1}
S_1(y,z) = &\int_{D_1} (a_{D_1}(x) - a_b(x))A(x)\nabla G_1(x,y)\cdot \nabla G_2(x,z) \diff x -\notag\\
&- \int_{D_1} ( q_{D_1}(x) - q_b(x))G_1(x,y)\: G_2(x,z) \diff x.
\end{align}
	
\begin{align}\label{S2}
S_2(y,z) = &\int_{D_2} (a_{D_2}(x) - a_b(x))A(x)\nabla G_1(x,y)\cdot \nabla G_2(x,z) \diff x -\notag\\
&- \int_{D_2} ( q_{D_2}(x) - q_b(x))G_1(x,y)\: G_2(x,z) \diff x.
\end{align}	

\noindent Set
\begin{equation}\label{f}
f(y,z) = S_1(y,z) - S_2(y,z).
\end{equation}
\bigskip

\noindent As a premise to the main result, we state the Proposition \ref{prop-bis} and Proposition \ref{prop5} that allow us to determine an upper bound of $f$ in terms of the Cauchy data $d(\mathcal{C}_1,\mathcal{C}_2)$ and a lower bound of $f$.

\begin{proposition}\label{prop-bis}
Let $D_1, D_2$ be two inclusions of $\Omega$ satisfying \eqref{inclusionbegin}-\eqref{inclusionend}. Let $\mathcal{C}_1, \mathcal{C}_2$ be the local Cauchy data associated with the inclusions $D_1, D_2$, respectively. Under the notation of Lemma \ref{metric-lemma}, let 
\[
y=h\nu(O), 
\]
where
\[
0<h\leq \bar{d}\left(1-\frac{\sin\theta_0}{4}\right) \mbox{ for }\,\theta_0=\arctan\frac{1}{L},
\]
and $\nu(O)$ is the outer unit normal of $D_1$ at $O$. For $\epsilon\in(0,1)$, if $d(\mathcal{C}_1,\mathcal{C}_2)<\epsilon$, it follows that
\begin{equation}\label{upper}
|f(y,y)|\leq c_1 \frac{\epsilon^{Bh^F}}{h^A}, 
\end{equation}
where $A, B, F, c_1>0$ constants that depend on the a priori data only.
\end{proposition}

\begin{proposition}\label{prop5}
Under the same assumptions as in Proposition \ref{prop-bis} and Lemma \ref{metric-lemma}, there exist $c_2>0, \bar{h}\in\left(0,\frac{1}{2}\right)$ that depend on the a priori data only such that 
\begin{equation}\label{lower}
|f(y,y)|\geq c_2 h^{2-n}-c_3 h^{2-2n}, \qquad 0<h<\bar{h}\bar{r}_2,
\end{equation}
where $y=h\nu(O)$ with $\bar{r}_2\in (0,\min\left\{\frac{r_0}{\sqrt{1+L^2}}\min\{1,L\},dist(O,D_2)\right\})$.
\end{proposition}

\begin{proof}[Proof of Theorem \ref{theorem}]
Let $O\in\p D_1$ be the point of Lemma \ref{metric-lemma} such that \eqref{uppermodist} is satisfied and w.l.o.g. assume that it coincides with the origin. Choose
\[
y_h=h\nu(O), \qquad 0<h<\bar{h}\bar{r}_2.
\]
Combining the upper bound \eqref{upper} and the lower bound \eqref{lower} at $y_h$, it follows that
\begin{equation*}
c_2 h^{2-n} - c_3 (dist(O,D_2)-h)^{2-2n} \leq c_1 \frac{\epsilon^{Bh^F}}{h^A},
\end{equation*}
where $c_1, c_2, c_3, A, B, F$ are the constants appearing in the Propositions \ref{prop-bis} and \ref{prop5} which depends on the a priori data only. 

\noindent Let $\epsilon_1\in (0,1)$ be such that $\exp(-B|\ln\epsilon_1|^{\frac{1}{2}})=\frac{1}{2}$. We distinguish between two cases.
\begin{itemize}
\item[a)] Assume that $\epsilon\in (0,\epsilon_1)$. Define $h=h(\epsilon)=\min\{|\ln\epsilon|^{-\frac{1}{2F}},dist(O,D_2)\}$. If $dist(O,D_2)\leq |\ln\epsilon|^{-\frac{1}{2F}}$ then from Lemma \ref{lemma1} and Lemma \ref{metric-lemma} the thesis follows straightforwardly. If $dist(O,D_2)\geq |\ln\epsilon|^{-\frac{1}{2F}}$, then $h=|\ln\epsilon|^{-\frac{1}{2F}}$ so that
\begin{equation}
c_4 (dist(O,D_2)-h)^{2(1-n)} \geq c_5 \big( 1- \epsilon^{Bh^F} h^{\tilde{A}}\big) h^{2-n},
\end{equation}
where $\tilde{A}=1-A$. Since
\[
\epsilon^{Bh^F}h^{\tilde{A}} \leq \exp(-B|\ln\epsilon|^{\frac{1}{2}}),
\]
then
\[
(dist(O,D_2)-h)^{2(1-n)} \geq c_6 h^{2-n},
\]
and since $h=|\ln\epsilon|^{-\frac{1}{2F}}$,
\[
dist(O,D_2) \leq c_7 |\ln\epsilon|^{-\eta}, \qquad \eta=\frac{n-2}{4F(n-1)}.
\]
\item[b)] Assume that $\epsilon\in [\epsilon_1,1)$, then, since $dist(O,D_2)\leq diam(\Omega)$, it follows that
\[
dist(O,D_2) \leq diam(\Omega) \frac{|\ln\epsilon|^{-\frac{1}{2F}}}{|\ln\epsilon_1|^{-\frac{1}{2F}}}.
\]
\end{itemize}

\end{proof}


\section{\large Proofs of technical propositions}\label{4}
In this section, we construct the Green functions associated with \eqref{greensystem} and we sketch the proofs of the Proposition \ref{prop-bis} and Proposition \ref{prop5} stated in section \ref{preliminaryprop}. In order to simplify the notation, we prefer to drop the pedix $i=1,2$ for all the quantities related to the two inclusions, so that the following statements hold true for generic inclusions $D$ satisfying \eqref{inclusionbegin}-\eqref{inclusionend}, coefficient $\sigma$ satisfying \eqref{condbegin}-\eqref{condend} and $q$ satisfying \eqref{zerorderbegin}-\eqref{zerorderend}. 

\subsection{\normalsize The construction of the Green function}
\begin{lemma}\label{lemma2}
For any $\sigma\in L^{\infty}(\Omega_0, Sym_n)$ that satisfies the uniform ellipticity condition, for any $q\in L^{\infty}(\Omega_0)$, $y\in \Omega_0$ there exists a unique distributional solution $G(\cdot,y)$ of the boundary value problem 
\begin{equation}\label{greensystem2}
\left\{ \begin{array}{ll}
\mbox{div}\left(\sigma(\cdot)\nabla G(\cdot,y)\right) + q(\cdot) G(\cdot,y) = - \delta(\cdot-y) &\textrm{$\textnormal{in}\quad \Omega_0$},\\
G(\cdot,y)=0 &\textrm{$\textnormal{on}\quad\p\Omega_0\setminus \Sigma_0$},\\
\sigma(\cdot) \nabla G(\cdot,y)\cdot \nu(\cdot) + i G(\cdot,y)=0 &\textrm{$\textnormal{on}\quad\Sigma_0$},
\end{array}
\right.
\end{equation}
such that for any $x,y\in \Omega_0, x\neq y$,
 there exists a constant $C$ depending on the a priori data such that
\begin{equation}
0<|G(x,y)|\leq C |x-y|^{2-n}.
\end{equation}
\end{lemma}

\begin{proof}
Our proof is based on the reasoning introduced in \cite[Proposition 3.1]{g.alessandrinim.v.dehoopr.gaburroe.sincich2018}. We find more convenient to divide the proof into two parts: in the first part we prove the well-posedness of the problem, whereas in the second part we construct a desired Green function. 

\paragraph{First step (well-posedness)} For $f\in L^2(\Omega_0)$, consider the following mixed boundary value problem
\begin{equation}\label{first}
\left\{ \begin{array}{ll}
\mbox{div}(\sigma\nabla v) + q\, v = f, &\textrm{$\textnormal{in}\quad \Omega_0$},\\
v = 0, &\textrm{$\textnormal{on}\quad\p\Omega_0\setminus \Sigma_0$},\\
\sigma(\cdot) \nabla v(\cdot)\cdot \nu(\cdot) + i v(\cdot)=0, &\textrm{$\textnormal{on}\quad\Sigma_0$}.
\end{array}
\right.
\end{equation}
Our goal is to find a solution $v\in H^1(\Omega_0)$ in the weak sense. Consider the adjoint mixed boundary value problem
\begin{equation}\label{second}
\left\{ \begin{array}{ll}
\mbox{div}(\sigma\nabla w) + q\, w = f, &\textrm{$\textnormal{in}\quad \Omega_0$},\\
w = 0, &\textrm{$\textnormal{on}\quad\p\Omega_0\setminus \Sigma_0$},\\
\sigma(\cdot) \nabla w(\cdot)\cdot \nu(\cdot) - i w(\cdot)=0, &\textrm{$\textnormal{on}\quad\Sigma_0$}.
\end{array}
\right.
\end{equation}
The Fredholm alternative tells us that existence of a solution to \eqref{first} implies uniqueness to \eqref{second} and viceversa (see \cite[Theorem 4, §6]{l.c.evans1997}). We prove uniqueness for both boundary value problems. Consider the homogeneous problem

\begin{equation}\label{third}
\left\{ \begin{array}{ll}
\mbox{div}(\sigma\nabla u) + q u = 0, &\textrm{$\textnormal{in}\quad \Omega_0$},\\
u = 0, &\textrm{$\textnormal{on}\quad\p\Omega_0\setminus \Sigma_0$},\\
\sigma(\cdot) \nabla u(\cdot)\cdot \nu(\cdot) \pm i u(\cdot)=0 &\textrm{$\textnormal{on}\quad\Sigma_0$}.
\end{array}
\right.
\end{equation}
Assume that $u\in H^1(\Omega_0)$. If we multiply \eqref{third} by $\bar{u}$ and integrate on $\Omega_0$, by the Green's identity it follows that
\begin{equation}
\int_{\Omega_0} \sigma(x)\nabla u(x)\cdot\nabla\bar{u}(x) \diff x - \int_{\Omega_0} q(x) |u(x)|^2 \diff x \pm i\int_{\Sigma_0} |u(x)|^2\diff x = 0.
\end{equation}
Therefore, $u=0$ on $\Sigma_0$. Hence, since the Neumann boundary condition becomes $\sigma(x)\nabla u(x)\cdot\nu(x)=0$ it follows that $u$ satisfies the Cauchy problem
\begin{equation}
\left\{ \begin{array}{ll}
\mbox{div}(\sigma\nabla u) + q u = 0, &\textrm{$\textnormal{in}\quad \Omega_0$},\\
u = 0, &\textrm{$\textnormal{on}\quad\p\Omega_0$}.
\end{array}
\right.
\end{equation}
so that $u=0$ in $\Omega_0$. In conclusion, we have proved existence and uniqueness for \eqref{first}. It remains to prove stability. For this purpose, consider $v\in H^1(\Omega_0)$ the weak solution to \eqref{first}, then by the weak formulation the following identities hold:
\begin{align}\label{vestimate}
\int_{\Sigma_0} |v|^2 &= -\Im\Big( \int_{\Omega_0} f\bar{v}\Big)\\
\int_{\Omega_0} \sigma(x)\,\nabla v(x)\cdot\nabla \bar{v}(x)\diff x &= -\Re\Big(\int_{\Omega_0} f\bar{v} \Big) + \int_{\Omega_0} q(x) |v(x)|^2\diff x
\end{align}
Define the following quantities:
\begin{equation*}
\epsilon^2 = \int_{\Sigma_0} |v|^2 + \int_{\Sigma_0} \sigma(x)\nabla v(x)\cdot \nu(x)\:\bar{v}(x) \diff x,
\end{equation*}
\begin{equation*}
\eta = \|f\|_{L^2(\Omega_0)}, \hspace{2cm} \delta = \|v\|_{L^2(\Omega_0)}, \hspace{2cm}
E = \|\nabla v\|_{L^2(\Omega_0)}.
\end{equation*}
From the Schwarz inequality and \eqref{vestimate}, it follows that
\begin{equation}
\int_{\Sigma_0} |v|^2 \leq \eta\:\delta,
\end{equation}
and combined with the impedance condition, 
\begin{equation}\label{epsestimate}
\epsilon^2 \leq 2 \eta\:\delta.
\end{equation}
From \eqref{vestimate}, we derive
\begin{equation}\label{Eestimate}
E^2 \leq \eta\,\delta + \|q\|^2_{L^{\infty}(\Omega_0)} \delta^2.
\end{equation}
Our claim is that there exists a positive constant which depends on the a priori data such that 
\begin{equation}
E^2\leq C \eta^2.
\end{equation}
We distinguish between two cases.
\begin{itemize}
\item If $\delta^2\leq \eta^2$, then the claim follows from \eqref{Eestimate}.
\item If $\delta^2\geq \eta^2$, we recall a quantitative estimate of unique continuation due to Carstea and Wang \cite[Theorem 5.3]{c.i.carsteaj.n.wang2020}, which is a generalization of \cite[Theorem 1.9]{g.alessandrinil.rondie.rossets.vessella2009}, so that
\begin{equation}
\delta^2 \leq \left(E^2 + \varepsilon^2 + \eta^2\right)\,\, \omega \Big(\frac{\varepsilon^2 + \eta^2}{E^2 + \varepsilon^2 + \eta^2} \Big)
\end{equation}
where $\omega(t)\leq C |\ln t|^{-\mu}$ for $t\in(0,1)$, $\omega(t)\rightarrow 0$ for $t\rightarrow 0^+$ and $C>0, \mu\in(0,1)$ positive constants depending on the a priori data only. From \eqref{epsestimate} and \eqref{Eestimate},
\begin{align*}
\delta^2 &\leq \left(\eta\delta + \|q\|_{L^{\infty}(\Omega_0)} \delta^2 + 2 \eta\delta + \eta^2\right) \,\,\omega \Big(\frac{\varepsilon^2 + \eta^2}{E^2 + \varepsilon^2 + \eta^2} \Big)\\
&\leq \left(4\delta^2 + \|q\|_{L^{\infty}(\Omega_0)} \delta^2 \right) \,\,\omega \Big(\frac{\varepsilon^2 + \eta^2}{E^2 + \varepsilon^2 + \eta^2} \Big)
\end{align*}
Multiplying by $\delta^2$ leads to
\begin{equation*}
1\leq (4 + \|q\|_{L^{\infty}(\Omega_0)})\,\, \omega\Big(\frac{\varepsilon^2 + \eta^2}{E^2 + \varepsilon^2 + \eta^2} \Big)
\end{equation*}
Inverting with respect to $\omega$ leads to
\begin{equation*}
\omega^{-1}\Big(\frac{1}{4 + \|q\|_{L^{\infty}(\Omega_0)}} \Big) \leq \frac{\varepsilon^2 + \eta^2}{E^2 + \varepsilon^2 + \eta^2}.
\end{equation*}
Set $C=\omega^{-1}\Big(\frac{1}{4 + \|q_D\|_{L^{\infty}(\Omega_0)}} \Big)$, it follows that
\begin{equation*}
 C E^2 \leq C( E^2 + \varepsilon^2 + \eta^2) \leq \varepsilon^2 + \eta^2 \leq 2 \eta \delta + \eta^2 \leq 3 \eta^2,
\end{equation*}
so that the claim follows.
\end{itemize}
Finally, by Poincarè inequality
\[
\|v\|_{L^2(\Omega_0)} \leq C \|\nabla v\|_{L^2(\Omega_0)}
\]
we can conclude that
\[
\|v\|_{H^1(\Omega_0)} \leq C \|f\|_{L^2(\Omega_0)}.
\]

\paragraph{Second step (construction of the Green function)} Fix $y\in \Omega_0$ and let $\tilde{G}(\cdot,y)$ be the weak solution to the boundary value problem
\begin{equation}\label{greensystem3}
\left\{ \begin{array}{ll}
\mbox{div}(\sigma(\cdot)\nabla \tilde{G}(\cdot,y)) = - \delta(\cdot-y) &\textrm{$\textnormal{in}\quad \Omega_0$},\\
\tilde{G}(\cdot,y)=0 &\textrm{$\textnormal{on}\quad\p\Omega_0\setminus \Sigma_0$},\\
\sigma(\cdot) \nabla \tilde{G}(\cdot,y)\cdot \nu(\cdot) + i \tilde{G}(\cdot,y)=0 &\textrm{$\textnormal{on}\quad\Sigma_0$}.
\end{array}
\right.
\end{equation}
From \cite{w.littmang.stampacchiah.w.weinberger1963}, $\tilde G(\cdot,y)$ satisfies the following properties:
\begin{equation}
\tilde G(x,y)=\tilde G(y,x),
\end{equation}
and
\begin{equation}
|\tilde G(x,y)|\leq C|x-y|^{2-n}, \qquad \text{for any} x\neq y, x,y\in\Omega_0.
\end{equation}
Fix $J=\lfloor\frac{n-1}{2} \rfloor$, for $x\in\Omega_0$, $x\neq y$, define
\begin{equation*}
\left\{ \begin{array}{ll}
R_0(x,y) = \tilde{G}(x,y)\\
\displaystyle R_j(x,y) = \int_{\Omega_0} q(z) \tilde{G}(x,z) R_{j-1}(z,y) \diff z, &\text{for }j=1,\dots,J.
\end{array}
\right.
\end{equation*}
The distribution $R_j(\cdot,y)$ is a weak solution to the boundary value problem
\begin{equation*}
\left\{ \begin{array}{ll}
\text{div}_x (\sigma(x)\nabla_x R_j(x,y)) = -q(x) R_{j-1}(x,y) &\text{for }x\in \Omega_0,\\
R_j(x,y)=0 &\text{for }x\in \p\Omega_0\setminus \Sigma_0,\\
\sigma(x) \nabla_x R_j(x,y)\cdot \nu(x) + i R_j(x,y) = 0 &\text{for }x\in \Sigma_0,
\end{array}
\right.
\end{equation*}
for $j=1,\dots,J$. From \cite[Chapter 2]{c.miranda1970} one can show that
\begin{equation}
|R_j(x,y)|\leq C |x-y|^{2j+2-n},\qquad \text{for every }j=0,1,\dots,J-1.
\end{equation}
For $j=J$, one has to distinguish between two cases:
\begin{itemize}
\item for $n$ even, 
\begin{equation}
|R_J(x,y)|\leq C(|\ln|x-y|| + 1);
\end{equation}
\item for $n$ odd,
\begin{equation}
|R_J(x,y)|\leq C
\end{equation}
\end{itemize}
where in both cases $C$ is a positive constant which depends on the a priori data only. In either cases,
\[
\|R_J(\cdot,y)\|_{L^p(\Omega_0)}\leq C, \qquad \text{for } 1\leq p<\infty.
\]
Define the distribution $R_{J+1}(\cdot,y)$, for $y\in\Omega_0$ as the weak solution to the boundary value problem
\begin{equation*}
\left\{ \begin{array}{ll}
\text{div}_x(\sigma(x)\nabla_x R_{J+1}(x,y)) + q(x) R_{J+1}(x,y) = -q(x) R_J(x,y) &\text{for }x\in\Omega_0\\
R_{J+1}(x,y)=0 &\text{for }x\in\p\Omega_0\setminus \Sigma_0\\
\sigma(x) \nabla_x R_{J+1}(x,y) \cdot \nu(x) + i R_{J+1}(x,y) = 0 &\text{for }x\in \Sigma_0.
\end{array}
\right.
\end{equation*}
It follows that $\|R_{J+1}(\cdot,y)\|_{H^1(\Omega_0)} \leq C$ where $C$ is a positive constant and by interior regularity estimates
\begin{equation}
|R_{J+1}(x,y)|\leq C, \quad \text{for }x\neq y,\,\, x,y\in\Omega_0.
\end{equation}
Define 
\begin{equation}
G(x,y) = \tilde{G}(x,y) + \sum_{j=1}^{J+1} R_j(x,y),
\end{equation}
For $y\in \Omega_0$, $G(\cdot,y)$ is a distributional solution to the boundary value problem \eqref{greensystem2} so that $G$ is the Green's function that we were looking for.

\paragraph{Third step (Simmetry of the Green function).}
Let $f,g\in C^{\infty}_0(\Omega_0)$. Let $u\in H^1(\Omega_0)$ be a weak solution to

\begin{equation}\label{u}
\left\{ \begin{array}{ll}
\mbox{div}(\sigma\nabla u) + q\, u = f, &\textrm{$\textnormal{in}\quad \Omega_0$},\\
u = 0, &\textrm{$\textnormal{on}\quad\partial\Omega_0\setminus \Sigma_0$},\\
\sigma(\cdot) \nabla u(\cdot)\cdot \nu(\cdot) + i u(\cdot)=0, &\textrm{$\textnormal{on}\quad\Sigma_0$}.
\end{array}
\right.
\end{equation}
Let $v\in H^1(\Omega_0)$ be a weak solution to
\begin{equation}\label{v}
\left\{ \begin{array}{ll}
\mbox{div}(\sigma\nabla v) + q\, v = g, &\textrm{$\textnormal{in}\quad \Omega_0$},\\
v = 0, &\textrm{$\textnormal{on}\quad\partial\Omega_0\setminus \Sigma_0$},\\
\sigma(\cdot) \nabla v(\cdot)\cdot \nu(\cdot) + i v(\cdot)=0, &\textrm{$\textnormal{on}\quad\Sigma_0$}.
\end{array}
\right.
\end{equation}
Let $G(\cdot,y)$ be the Green function solution to \eqref{greensystem2}. The weak solution $u$ of \eqref{u} can be written as
\[
u(x) = \int_{\Omega_0} G(x,y) f(y) \diff y,
\]
and similarly,
\[
v(x) = \int_{\Omega_0} G(x,y) g(y) \diff y.
\]
Hence, by the Green's identity, it follows that
\begin{align*}
\int_{\Omega_0} u(x)\,g(x) \diff x = \int_{\Omega_0} f(x)\,v(x) \diff x.
\end{align*}
Hence,
\begin{equation}
\int_{\Omega_0} \left[\int_{\Omega_0} G(x,y) f(y) \diff y \right] g(x) \diff x = \int_{\Omega_0} \left[\int_{\Omega_0} G(y,x) g(x) \diff x \right] f(y) \diff y,
\end{equation}
By Fubini's theorem and the arbitrarity of $f$ and $g$, it follows that
\[
G(x,y) = G(y,x)\qquad \text{for any }x,y\in \Omega_0.
\]
\end{proof}

\bigskip

\subsection{Upper bound for the function $f$}
Let $\Omega\subset\R^n$, $n\geq 3$ be a bounded domain satisfying \eqref{omegabegin}-\eqref{omegaend} and let $D_1$, $D_2$ be two inclusions of $\Omega$ satisfying \eqref{inclusionbegin}-\eqref{inclusionend}. Let $\sigma_{1}$, $\sigma_{2}$ be the anisotropic conductivities satisfying \eqref{condbegin}-\eqref{condend} and let $q_1$, $q_2$ be the coefficients of the zero order term satisfying \eqref{zerorderbegin}-\eqref{zerorderend}. Before proving Proposition \ref{prop-bis}, let us introduce some useful formulas.

\noindent Let $u_j\in H^1(\Omega)$ with $j=1,2$ be a weak solution to the Dirichlet problem
\begin{equation}\label{dirproblem}
\begin{cases}
\text{div}(\sigma_j\nabla u_j) + q_j u_j = 0 &\text{in }\Omega,\\
u_j|_{\p\Omega}\in H^{\frac{1}{2}}_{00}(\Sigma).
\end{cases}
\end{equation}
\noindent Integrating by parts \eqref{dirproblem} leads to the following identity,
\begin{equation}\label{claim1}
\int_{\Omega} (\sigma_2-\sigma_1)\nabla u_1\cdot \nabla u_2 + \int_{\Omega} (q_1-q_2) u_1 u_2 = 
\langle \sigma_2\nabla\bar{u}_2\cdot \nu\vert_{\p\Omega}, u_1\rangle - \langle \sigma_1\nabla u_1\cdot\nu\vert_{\p\Omega}, \bar{u}_2\rangle.
\end{equation}

\noindent Let $v_j$ for $j=1,2$ with $v_j\in H^1(\Omega)$ be weak solution to $\text{div}(\sigma_j\nabla v_j) + q_j v_j = 0$ in $\Omega$.
\begin{equation}\label{claim2}
\langle \sigma_j \nabla v_j\cdot\nu\vert_{\p\Omega},\bar{u}_j\rangle - \langle \sigma_j \nabla\bar{u}_j\cdot\nu\vert_{\p\Omega}, v_j\rangle = 0,\qquad \text{for }j=1,2.
\end{equation}
Moreover, from \eqref{claim1} and \eqref{claim2}, one can show that
\begin{multline}\label{claim}
\int_{\Omega} (\sigma_2-\sigma_1)\nabla u_1\cdot \nabla u_2 + \int_{\Omega} (q_1-q_2) u_1 u_2 = \\
=\langle \sigma_2\nabla\bar{u}_2\cdot\nu\vert_{\p\Omega}, (u_1-v_2)\rangle - \langle \sigma_1\nabla u_1\cdot\nu\vert_{\p\Omega} - \sigma_2\nabla v_2\cdot\nu\vert_{\p\Omega},\bar{u}_2\rangle.
\end{multline}

\noindent Finally, \eqref{claim} and the the Cauchy-Schwarz inequality allow us to bound $f(y,y)$ with $d(\mathcal{C}_1,\mathcal{C}_2)$ as
\begin{equation}\label{claim3}
\left|\int_{\Omega} (\sigma_2-\sigma_1)\nabla u_1\cdot \nabla u_2 + \int_{\Omega} (q_1-q_2) u_1 u_2\right| \leq d(\bar{\mathcal{C}}_1,\bar{\mathcal{C}}_2)\,\, \|(u_1,\sigma_1\nabla u_1\cdot\nu)\|_{\mathcal{H}}\,\, \|(\bar{u}_2,\sigma_2\nabla \bar{u}_2\cdot \nu)\|_{\mathcal{H}}.
\end{equation}

\noindent We introduce the asymptotic estimates for the gradient of the Green function $G$ that will be used in the proof of Theorem \ref{theorem}.

\begin{proposition}\label{prop1}
Let $\Omega$, $D$ be, respectively, a bounded domain satisfying \eqref{omegabegin}-\eqref{omegaend} and an inclusion satisfying \eqref{inclusionbegin}-\eqref{inclusionend}. Then there exists a positive constant $C_1$ that depends on the a priori data only such that
\begin{align}
|\nabla_x G(x,y)| &\leq C_1 |x-y|^{1-n},
\end{align}
for any $x,y\in \R^3$.
\end{proposition}

\begin{proof}[Proof of Proposition \ref{prop1}]
See \cite[Proposition 3.4]{g.alessandrinim.dicristo2005}.
\end{proof}

\begin{proof}[Proof of Proposition \ref{prop-bis}.]
Fix a point $\bar{y}\in D_0$ such that $dist(\bar{y},\p\Omega)\geq \tilde cr_0$, for $0<\tilde c<1$ suitable constant. For any $\bar{w}\in \R^n\setminus\bar{\Omega}_D$, $f(\bar{y},\cdot)$ is a weak solution to
\begin{equation}\label{eqA}
\mbox{div}_w\left(a_b(\cdot)A(\cdot)\nabla_w f(\bar{y},\cdot)\right) +q_b(\cdot) f(\bar{y},\cdot) = 0 \qquad \text{in }\R^n\setminus\bar{\Omega}_D.
\end{equation}
First, choose $\bar{w}\in D_0$ such that $dist(\bar{w},\p\Omega)\geq \tilde cr_0$, for $0<\tilde c<1$. By \eqref{claim3}, 
\[
|f(\bar{y},\bar{w})| \leq d(\bar{\mathcal{C}}_1,\bar{\mathcal{C}}_2)\,\, \|(G_1(\cdot,\bar{y}),\sigma_1\nabla G_1(\cdot,\bar{y})\cdot\nu)\|_{\mathcal{H}}\,\, \|(G_2(\cdot,\bar{w}),\sigma_2\nabla \bar{G}_2(\cdot,\bar{w})\cdot \nu)\|_{\mathcal{H}}.
\]
Since by Proposition \ref{prop1} and the definition of the norm on $\mathcal{H}$,
\begin{align*}
\|(G_1(\cdot,\bar{y}),\sigma_1\nabla G_1(\cdot,\bar{y})\cdot\nu)\|_{\mathcal{H}} \leq \left(\|G_1(\cdot,\bar{y})\|^2_{H^{1}(\p\Omega)} + C \|\nabla G_1(\cdot,\bar{y})\|^2_{H^{1}(\p\Omega)} \right)^{\frac{1}{2}},
\end{align*}
it follows that
\begin{equation}\label{eqB}
|f(\bar{y},\bar{w})|\leq C_1\epsilon.
\end{equation}
where $C_1$ is a positive constant depending on the a priori data only.

\noindent Now let  $\bar{w}\in \bar{\Omega}^{r_0}\setminus \bar{\Omega}_D$ where $\Omega^{r_0}=\{x\in\R^3\,:\,dist(x,\p\Omega)> r_0\}$ and let $\bar{y}\in D_0$. By Proposition \ref{prop1} and since $|x-\bar{y}|\geq r_0$,
\begin{align*}
|f(\bar{y},\bar{w})| &\leq C \sum_{j=1}^2 \int_{D_j} |x-\bar{y}|^{1-n} |x-\bar{w}|^{1-n} \diff x\\
&\leq C \sum_{j=1}^2 \int_{D_j} |x-\bar{w}|^{1-n} \diff x.
\end{align*}
Choose $\tilde R = diam(\Omega) + r_0\leq C r_0$, where $C$ is a constant which depends only on $L$. Hence, $B_{\tilde R}(\bar{w})\subset \Omega$ and for $j=1,2$,
\[
\int_{D_j} |x-\bar{w}|^{1-n} \diff x \leq \int_{B_{\tilde R}(\bar{w})} |x-\bar{w}|^{1-n} \diff x \leq C.
\]

\noindent The next step consists in the determination of an estimate for $f(\bar{y},w)$ when $w\in \de$. For $h>0$, define
$$(\G)^h = \{x\in\Omega^c_D\,:\,\text{dist}(x,\p\Omega_D)\geq h\}.$$
\noindent For $w\in (\de)^h$, by Proposition \ref{prop1},
\begin{align*}
|S_1(\bar{y},w)| \leq C \int_{D_1} |x-\bar{y}|^{1-n} |x-w|^{1-n} \diff x \leq C h^{1-n}.
\end{align*}
Similarly, $|S_2(\bar{y},w)|\leq C h^{1-n}$, so that
\begin{equation}\label{eqC}
|f(\bar{y},w)|\leq C h^{1-n}.
\end{equation}

\noindent We proceed with a quantitative estimate of propagation of smallness for the function $f$ with respect to the second variable. Let $O$ be the point of Lemma \ref{metric-lemma} and assume to be in a coordinate system in which $O$ coincides with the origin and set $y_h=h\nu(O)$. The goal is to propagate \eqref{eqB} inside $\mathcal{G}$ up to $y_h$. 


\noindent In order to do it, fix $\bar{y}, w\in D_0$ such that $dist(\bar{y},\p\Omega)\geq r_0$ and $dist(w,\p\Omega)\geq r_0$. By Lemma \ref{metric-lemma} we know that there exists a curve $\gamma\subset (\overline{\Omega^{r_0}}\cup D_0)\setminus \bar{\Omega}_D$ joining $w$ to the point $Q=\bar{d} \nu(O)$, where $\nu(O)=-e_n$, such that $V(\gamma)\subset \R^3\setminus \Omega_D$ with $R=\frac{\bar{d}}{\sqrt{1+L^2}}$ and $\theta_0=\arcsin \frac{R}{\bar{d}}$.

\noindent Notice that, since $f(\bar{y},\cdot)$ is a weak solution to \eqref{eqA}, one can apply the three sphere inequality in the ball $B_{r_0}(\bar{x})$, where the point $\bar{x}\in D_0$ is such that $dist(\bar{x},\p\Omega)=\frac{3r_0}{4}$. Choose $r=\frac{r_0}{4}$, then, for radii $r, 3r, 4r$, the following estimate holds,
\begin{equation}\label{eqD}
\|f(\bar{y},\cdot)\|_{L^{\infty}(B_{3r}(\bar{x}))}\leq C \|f(\bar{y},\cdot)\|_{L^{\infty}(B_{r}(\bar{x}))}^{\tau} \|f(\bar{y},\cdot)\|_{L^{\infty}(B_{4r}(\bar{x}))}^{1-\tau}
\end{equation}
where
\[
\tau = \frac{\ln \frac{4\bar{\lambda}}{3}}{\ln \frac{4\bar{\lambda}}{3} + c \ln \frac{3}{\bar{\lambda}}},
\]
 $0<\tau<1$ and $C>0$ depends on $\bar{\lambda}, L, r_0$.

\noindent Consider $w, Q$ and the curve $\gamma$ as above. We select a finite number of points on $\gamma$ as follows. Set $\phi_1=w$. Then
\begin{enumerate}
\item if $|\phi_{j-1}-Q|>r$, then set $\phi_j=\gamma(t_j)$ where $t_j=\max\{t: |\gamma(t)-\phi_{j-1}|=r\}$;
\item otherwise, set $s=j$, $\phi_s=Q$ and stop the process.
\end{enumerate}

\noindent By iterating the three sphere inequality along the chain of balls centred at $\phi_j$ for $j=1,\dots,s$, and assuming that $s\leq S$ where $S$ depends on $n$ only, one derives that for any $r_1$ with $0<r_1<r$, 
\[
\|f(\bar{y},\cdot)\|_{L^{\infty}(B_{\frac{r_1}{2}}(Q))}\leq C \|f(\bar{y},\cdot)\|_{L^{\infty}(B_{\frac{r_1}{2}}(\bar{w}))}^{\tau^s} \|f(\bar{y},\cdot)\|_{L^{\infty}(\mathcal{G})}^{1-\tau^s}
\]
By \eqref{eqB} and \eqref{eqC},
\begin{equation}\label{eqF}
\|f(\bar{y},\cdot)\|_{L^{\infty}(B_{\frac{r_1}{2}}(Q))} \leq C \epsilon^{\tau^S} (h^{1-n})^{1-\tau^S}.
\end{equation}
The goal is to propagate the smallness from $Q$ to $y_h$. Consider the truncated cone $C(O, -e_n, d, \theta_0)$ where $d=\frac{\bar{d}^2-R^2}{\bar{d}}$. Define
\[
\lambda_1 = \min\left\{ \frac{d}{1+\sin\theta_0}, \frac{d}{3\sin\theta_0}\right\}, \quad \tilde\theta_0=\arcsin\left(\frac{\sin\theta_0}{8} \right)
\]
\[
w_1=O+\lambda_1\nu(O),\quad \rho_1=\lambda_1\sin\theta_1,\quad a=\frac{1-\sin\theta_1}{1+\sin\theta_1},
\]
so that $B_{\rho_1}(w_1)\subset C(O,\nu(O), \theta_1, d)$ and $B_{4\rho_1}(w_1)\subset C(O,\nu(O),\tilde\theta_0,d)$. Since $\rho_1<\frac{r_0}{2}$, one can apply \eqref{eqF} in the cone $C(O,-e_n,\theta_1,d)$ over a chain of balls of shrinking radii $\rho_k=a\rho_{k-1}$ centred at points $w_k=O+\lambda_k\nu(O)$ with $\lambda_k=a\lambda_{k-1}$. Denote by $d(k)=|w_k-O|-\rho_k$, then $d(k)=a^{k-1}d(1)$. We consider $h\leq d(1)$ and define a natural number $k(h)$ as the smallest positive integer such that $d(k(h))\leq h$. Hence
\[
\frac{\left|\ln\frac{h}{d(1)}\right|}{|\ln a|} \leq k(h)-1 \leq \frac{\left|\ln\frac{h}{d(1)}\right|}{|\ln a|} +1.
\]
By iterating the three-sphere inequality over the chain of balls $B_{\rho_1}(w_1),\dots,B_{\rho_{k(h)}}(w_{k(h)})$, one derives
\begin{equation}\label{eqG}
\|f(\bar{y},\cdot)\|_{L^{\infty}\left(B_{\rho_{k(h)}}(w_{k(h)})\right)} \leq c (h^{1-n})^{A''} \epsilon^{\beta\tau^{k(h)-1}},
\end{equation}
where $\beta=\tau^S$ and $A''=1-\beta$.

\noindent Consider now $f(y,w)$ as a function of $y$. Notice that for any $w\in \R^n\setminus \Omega_D$, $f(\cdot,w)$ is a weak solution to
\[
div_y(a_b(\cdot)A(\cdot)\nabla_y f(\cdot,w)) + q_b(\cdot)f(\cdot,w)=0\qquad \text{in }\R^n\setminus\Omega_D.
\]
For any $y,w\in \mathcal{G}^h$, by Proposition \eqref{prop1},
\[
|S_1(y,w)|\leq c \int_{D_1} |x-y|^{1-n} |x-z|^{1-n} \diff x \leq c h^{2(1-n)}.
\]
Similarly, $|S_2(y,w)|\leq c h^{2(1-n)}$, so that
\[
|f(y,w)|\leq c h^{2(1-n)}. \qquad \text{for any }y,w \in \mathcal{G}^h.
\]
\noindent Now, for $y\in D_0$ such that $dist(y,\p\Omega)\geq \tilde c r_0$, for $w\in \mathcal{G}^h$ and by \eqref{eqG},
\[
|f(y,w)|\leq c (h^{1-n})^{A''} \epsilon^{\beta\tau^{k(h)-1}},
\]
where $A'', \beta$ are defined as above. Fix $w\in \mathcal{G}$ such that $dist(w,\Omega_D)=h$, $\bar{y}\in D_0$ such that $dist(\bar{y},\p\Omega)\geq \frac{3r_0}{2}$, then for $\bar{r}=\frac{r_0}{2}$, $3\bar{r}$, $4\bar{r}$ and $y_1=w_1$ defined as above, by an iterated application of the three sphere inequality one derives
\begin{align*}
\|f(\cdot,w)\|_{L^{\infty}\left(B_{\bar{r}}(y_1)\right)} &\leq c \|f(\cdot,w)\|_{L^{\infty}(B_{\bar{r}}(\bar{y}))}^{\tau^s} \|f(\cdot,w)\|_{L^{\infty}(\mathcal{G})}^{1-\tau^s}\\
&\leq c (h^{2-2n})^{A''} \epsilon^{\beta^2\tau^{k(h)-1}},
\end{align*}
where $A'=1-\beta+A''\tau^s$, $\beta=\tau^S$. Once more we apply the three sphere inequality inside the cone of vertex $O$ over a chain of balls with shrinking radii as above so that
\[
\|f(\cdot,w)\|_{L^{\infty}\left(B_{\rho_{k(h)}}(y_{k(h)})\right)} \leq c (h^{A'})^{1-\tau^{k(h)-1}} (\epsilon^{\beta^2\tau^{k(h)-1}})^{\tau^{k(h)-1}}.
\]
Now, if we choose $y=w=y_h$, one derives
\[
|f(y_h,y_h)| \leq c h^{-A} (\epsilon^{\beta^2\tau^{k(h)-1}})^{\tau^{k(h)-1}}.
\]
where $A=-(2-2n)A'(1-\tau^{k(h)-1})>0$. Since $k(h)\leq c|\ln h|=-c\ln h$, then
\[
\tau^{k(h)} = e^{-c\ln h \ln \tau} = h^{-c\ln\tau} = h^F, \qquad \text{where } F=c|\ln h|.
\]
In conclusion,
\begin{align*}
|f(y,y)| \leq c_1 h^{-A} \epsilon^{\beta^2\tau^{2(k(h)-1)}} = c_1 h^{-A} e^{\beta^2\tau^{2(k(h)-1)}\ln\epsilon} = c_1 h^{-A}\epsilon^{Bh^F},
\end{align*}
where $B=\beta^2$.
\end{proof}

\subsection{Lower bound for the function $f$}

In view of the proof Proposition \ref{prop5}, we introduce the asymptotic estimates for the Green's functions which are solutions of \eqref{greensystem} with respect to two auxiliary families of Green's functions.
 
\noindent First, let $P\in \p D_1\cap \p \Omega_D$ be the point of Lemma \ref{metric-lemma}.  Up to a rigid transformation, we can assume that $P$ coincides with the origin $O$ and
\[
D\cap Q_{\frac{r_0}{3}} = \left\{ x\in Q_{\frac{r_0}{3}}\,: \,x_n \geq \varphi(x')\right\},
\]
where $\varphi\in C^2(B'_{\frac{r_0}{3}})$.

\noindent Following the lines of \cite[Theorem 4.2]{g.alessandrinis.vessella2005}, we introduce a change of coordinates which flattens the boundary near $O$. Let $\tau\in C^{\infty}(\R)$ such that $0\leq \tau(s)\leq 1$, $\tau(s)=1$ for $s\in (-1,1)$ and $\tau(s)=0$ for $s\in \R\setminus(-2,2)$ and $|\tau'(s)|\leq 2$ for any $s\in \R$. Set
\[
r_1=\frac{r_0}{3}\min\left\{\frac{1}{2}(8L)^{-1}, \frac{1}{4}\right\}.
\]
The following change of coordinates 
\begin{align*}
\xi=\phi(x)=
\begin{cases}
\xi'=x'\\
\xi_n=x_n-\varphi(x')\tau\left(\frac{|x'|}{r_1}\right)\tau\left(\frac{x_n}{r_1} \right),
\end{cases}
\end{align*}
is  a $C^{1,1}$ diffeomorphism of $\R^n$ into itself and allows us to flatten locally the boundary of the inclusion. In what follows, we keep the notation with $x$, since the exponent appearing in the asymptotic estimates does not depend on the change of coordinates.

\noindent Set
\begin{equation}\label{sigma0}
\sigma_0(x)=\left(a^- + (a^+ - a^-)\chi^+\right)A \quad \text{and}\quad q_0(x)=q_b(0)+\left(q_{D_i}(0) - q_b(0)\right)\chi^+(x)
\end{equation}
where 
$$a^-=a_b(0), \quad a^+=a_D(0), \quad A=A(0), \quad \chi^+=\chi_{\R^n_+}.$$
For $y\in D_0$, let $G_0(\cdot,y)$ be the weak solution to
\begin{equation*}
\begin{cases}
\text{div}(\sigma_0(x)\nabla G_0(x,y)) + q_0(x) G_0(x,y)=-\delta(x-y), &\text{for }x\in\Omega_0,\\
G_0(x,y)=0, &\text{for }x\in \p\Omega_0\setminus\Sigma_0,\\
\sigma_0(x)\nabla G_0(x,y)\cdot\nu(x) + iG_0(x,y)=0, &\text{for }x\in \Sigma_0.
\end{cases}
\end{equation*}
Let $H$ be the fundamental solution to
$$\text{div}\left(\sigma_0(\cdot)\nabla H(\cdot,y)\right)=-\delta(\cdot-y).$$
Recalling \cite{s.foschiattie.sincichr.gaburro2021}, $H$ has the following expression

\begin{equation}\label{fundsol}
H(x,y) = |J|
\begin{cases}
\displaystyle \frac{1}{a^+}\Gamma(Jx,Jy) + \frac{a^+-a^-}{a^+(a^++a^-)}\Gamma(Jx,Jy^*) &\text{if }x_n, y_n >0,\\
\displaystyle \frac{2}{a^-+a^+} \Gamma(Jx,Jy) &\text{if } x_n\cdot y_n<0,\\
\displaystyle \frac{1}{a^-}\Gamma(Jx,Jy) + \frac{a^--a^+}{a^-(a^++a^-)}\Gamma(Jx,Jy^*) &\text{if }x_n, y_n <0,
\end{cases}
\end{equation}
where $y^* = (y_1,\dots,y_{n-1},-y_n)$, $J=\sqrt{A(0)^{-1}}$ and $|J|=\text{det}(\sqrt{A(0)^{-1}})$.

\noindent We introduce the asymptotic estimates for the Green's functions $G$ with respect to $H$.
\begin{proposition}\label{prop2}
Under the same assumptions as in Proposition \ref{prop1}, there exists positive constants $C_2, C_3$ and $\theta_1\in(0,1)$ that depend on the a priori data only such that
\begin{align}
|G(x,y) - H(x,y)| &\leq C_2 |x-y|^{3-n},\\
|\nabla_x G(x,y) - \nabla_x H(x,y)| &\leq C_3 |x-y|^{1-n+\theta_1},
\end{align}
for every $x\in D\cap B_r$ and $y=h\nu(O)$ where $r\in \left(0,\frac{r_0}{2}\min\{(8L)^{-1},\frac{1}{4}\}\right)$ and $h\in \left(0,\frac{r_0}{4}\min\{(8L)^{-1},\frac{1}{4}\}\right)$.
\end{proposition}

\begin{proof}[Proof of Proposition \ref{prop2}]
Notice that for $x, y$ as in the assumptions,
\[
|G(x,y) - H(x,y)| \leq |G(x,y)-G_0(x,y)| + |G_0(x,y)-H(x,y)|.
\]
Hence we can split the proof of Proposition \ref{prop2} into two claims.

\begin{claim}\label{propstar}
There exists positive constants $C_4, C_5$ and $\theta_1\in(0,1)$ that depend on the a priori data only such that
\begin{align}
|G(x,y) - G_0(x,y)| &\leq C_4 |x-y|^{3-n},\\
|\nabla_x G(x,y) - \nabla_x G_0(x,y)| &\leq C_5 |x-y|^{1-n+\theta_1},
\end{align}
for every $x\in D\cap B_r$ and $y=h\nu(O)$ where $r\in \left(0,\frac{r_0}{2}\min\{(8L)^{-1},\frac{1}{4}\}\right)$ and $h\in \left(0,\frac{r_0}{4}\min\{(8L)^{-1},\frac{1}{4}\}\right)$.
\end{claim}

\begin{proof}[Proof of Claim \ref{propstar}]
We follow the lines of the proof of \cite[Proposition 3.4]{g.alessandrinim.dicristo2005}. For simplicity, we consider a generic inclusion $D$ with jump coefficients $\sigma$, $q$. Fix $P\in\p D$ so that under a suitable transformation of coordinates, $P=O$. Let $G$ denote the Green function associated with the elliptic operator div$(\sigma(\cdot)\nabla\cdot)+q(\cdot)$ so that for any $y\in \Omega_0$, $G(\cdot,y)$ is a distributional solution to the following boundary value problem
\begin{equation}\label{system1}
\begin{cases}
\text{div}(\sigma(x)\nabla G(x,y)) + q(x)G(x,y)=-\delta(x-y), &\text{for }x\in \Omega_0,\\
G(x,y)=0, &\text{for }x\in\p\Omega_0\setminus\Sigma_0,\\
\sigma(x)\nabla G(x,y)\cdot\nu(x) + iG(x,y)=0, &\text{for }x\in \Sigma_0.
\end{cases}
\end{equation}
For $O\in \p D$, let $\sigma_0, q_0$ be as in \eqref{sigma0}. For $y\in \Omega_0$, let $G_0(\cdot,y)$ be the Green's function which is a distributional solution to the following auxiliary boundary value problem
\begin{equation}\label{system2}
\begin{cases}
\text{div}(\sigma_0(x)\nabla G_0(x,y)) + q_0(x) G_0(x,y)=-\delta(x-y), &\text{for }x\in\Omega_0,\\
G_0(x,y)=0, &\text{for }x\in \p\Omega_0\setminus\Sigma_0,\\
\sigma_0(x)\nabla G_0(x,y)\cdot\nu(x) + iG_0(x,y)=0, &\text{for }x\in \Sigma_0.
\end{cases}
\end{equation}
Define 
\begin{equation}
R(x,y) = G(x,y) - G_0(x,y).
\end{equation}
Subtracting the first equation of \eqref{system2} to \eqref{system1}, it follows that $R(x,y)$ is a weak solution in $\Omega_0$ to the equation
\begin{equation}
\text{div}(\sigma(\cdot)\nabla R(\cdot,y)) + q(\cdot) R(\cdot,y) = - \text{div}\left((\sigma(\cdot)-\sigma_0(\cdot))\nabla G_0(\cdot,y)\right)  + [q_0(\cdot)-q(\cdot)] G_0(\cdot,y), \quad \text{in }\Omega_0
\end{equation}
with boundary conditions
\begin{align*}
\left\{
\begin{array}{ll}
R(x,y) = 0 &\text{for }x\in \p\Omega_0\setminus \Sigma_0,\\
\displaystyle \sigma_0(x)\nabla_x R(x,y) \cdot\nu(x) + i R(x,y) = (\sigma_0(x)-\sigma(x)) G(x,y) &\text{for }x\in \Sigma_0.
\end{array}
\right.
\end{align*}

\noindent Then the following representation formula holds
\begin{align}\label{restoR}
-R(x,y) &= \int_{\Omega_0} \big(\sigma(z) - \sigma_0(z)\big) \nabla_z G(z,x)\cdot \nabla_z G_0(z,y) \diff z +\notag\\
&+ \int_{\Omega_0} \big(q(z) - q_0(z)\big) G(z,x) G_0(z,y) \diff z\notag\\
&+\int_{\Sigma_0} \left[\sigma_0(z)\nabla_z G_0(z,y)\cdot\nu G(z,x) - \sigma(z)\nabla_z G(z,x)\cdot \nu G_0(z,y)\right]\diff S(y)
\end{align}
The boundary integral are bounded (for instance by the Schwarz inequality and the trace estimates). The second volume integral in \eqref{restoR} is less singular than the first volume integral, so that we find convenient to study the first volume integral. Let us split the domain of integration into the union of the subdomains $\Omega\cap Q_{\bar{r}_0}$ and $\Omega\setminus Q_{\bar{r}_0}$. For $x\in \Omega\cap Q_{\bar{r}_0}$, 
\[
|\sigma(z)-\sigma_0(z)|\leq C|z|.
\]
Then we can apply the same argument of \cite[Proposition 4.1]{m.dicristo2009} and conclude that 
\begin{equation}\label{Restimate}
|R(x,y)|\leq C_4 |x-y|^{3-n},
\end{equation}
where $\tilde{C}_1$ is a positive constant which depends only on the a priori data. 

\noindent Regarding the estimate for the gradient of the residual $R$, recalling that the boundary of $D$ is $C^2$ and hence $C^{1,1}$, for $x\in D\subset B_r$, we consider a cube $Q\subset B^+_{\frac{r}{4}}$ of side $\frac{c\bar{r}_0}{4}$, $c\in (0,1)$ so that $y\notin Q$ and $x\in \p Q$. By \cite[Lemma 3.2]{g.alessandrini1990}, the following interpolation formula holds
\begin{equation}\label{interpolationA}
\|\nabla R(\cdot,y)\|_{L^{\infty}(Q)} \leq C \|R(\cdot,y)\|_{L^{\infty}(Q)}^{\frac{1}{2}} |\nabla_x R(\cdot,y)|^{\frac{1}{2}}_{1,Q},
\end{equation}
where $C$ depends on $L$ only. For $y=h\nu(O)$ for $h$ as in the Proposition statement, from the piecewise H\"older continuity of $\nabla_x G(x,y)$ and $\nabla_x G_0(x,y)$ (see \cite[Theorem 16.2]{o.ladyzhenskayan.uraltseva1968}),
\[
|\nabla_x G(\cdot,y)|_{1,Q}, |\nabla_x G_0(\cdot,y)|_{1,Q} \leq C h^{-n}.
\]
Therefore,
\begin{equation}\label{Rsemiestimate}
|\nabla_x R(\cdot, y)|_{1,Q} \leq C h^{-n}
\end{equation}
and collecting \eqref{Restimate}, \eqref{interpolationA} and \eqref{Rsemiestimate}, it follows that
\[
|\nabla_x R(x,y)|\leq C_5 |x-y|^{1-n+\theta_1},
\]
where $\theta_1=\frac{1}{2}$.

\end{proof}

\begin{claim}\label{prop3}
There exists positive constants $C_6, C_7$ that depend on the a priori data only such that
\begin{align}
|G_0(x,y) - H(x,y)| &\leq C_6 |x-y|^{4-n},\\
|\nabla_x G_0(x,y) - \nabla_x H(x,y)| &\leq C_7 |x-y|^{2-n},
\end{align}
for every $x\in D\cap B_r$ and $y=h\nu(O)$ where $r\in \left(0,\frac{r_0}{2}\min\{(8L)^{-1},\frac{1}{4}\}\right)$ and $h\in \left(0,\frac{r_0}{4}\min\{(8L)^{-1},\frac{1}{4}\}\right)$.
\end{claim}

\begin{proof}[proof of Claim \ref{prop3}]
We follow the argument in \cite[Proposition 4.2]{m.dicristo2009}. Let $y,z\in \Omega_0$ and let $D$ be an inclusion in $\Omega$ satisfying \eqref{inclusionbegin}-\eqref{inclusionend}. Recall that $G_0(\cdot,y)$ is the weak solution to 
\begin{equation*}
\begin{cases}
\text{div}(\sigma_0(x)\nabla G_0(x,y)) + q_0(x) G_0(x,y)=-\delta(x-y), &\text{for }x\in\Omega_0,\\
G_0(x,y)=0, &\text{for }x\in \p\Omega_0\setminus\Sigma_0,\\
\sigma_0(x)\nabla G_0(x,y)\cdot\nu(x) + iG_0(x,y)=0, &\text{for }x\in \Sigma_0.
\end{cases}
\end{equation*}
and $H(\cdot,y)$ is the fundamental solution to
\begin{equation}
\text{div}_x (\sigma_0(\cdot)\nabla_x H(\cdot,y)) = -\delta(\cdot-y) \qquad \text{in }\R^n.
\end{equation}
The residual function 
\[
R(x,y)=G_0(x,y)-H(x,y)
\]
is a weak solution to the equation
\[
\begin{cases}
\text{div}_x (\sigma_0(\cdot)\nabla_x R(\cdot,y)) = - q_0(\cdot) G_0(\cdot,y)  &\text{in }\Omega_0\\
R(\cdot,y)=-H(\cdot,y) &\text{on }\p\Omega_0\setminus \Sigma_0\\
\sigma_0(\cdot)\nabla R(\cdot,y)\cdot \nu(\cdot) + i R(\cdot,y)= -\sigma_0(\cdot)\nabla H(\cdot,y)\cdot \nu(\cdot) - i H(\cdot,y) &\text{on }\Sigma_0.
\end{cases}
\]
Its representation formula is 
\begin{align}
-R(x,y) = & \int_{\Omega_0} q_0(z) G_0(z,x) H(z,y) \diff x +\notag\\
&+\int_{\p\Omega_0} \sigma_0(z)\left[\nabla_z H(z,x)\cdot \nu H(z,y) - \nabla_z G_0(z,x)\cdot\nu H(z,y) \right] \diff S(z)+\notag\\
&+\int_{\p\Omega_0} \left[\nabla_z H(z,y)\cdot\nu G_0(z,x) - \nabla_z H(z,y)\cdot\nu H(z,x)\right]\diff S(z).
\end{align}
The surface integral can be easily bounded from above using Cauchy-Schwarz inequality by
 a constant that depends on the a priori data only. Regarding the volume integral, by \eqref{boundgreen} it follows that
\begin{align*}
\Big|\int_{\Omega} q_0(z) G_0(z,x) H(z,y) \diff x \Big| &\leq \|q_0\|_{L^{\infty}(\Omega)} \int_{\Omega} |G_0(z,x)| |H(z,y)| \diff z\\
&\leq C \int_{\Omega} |z-x|^{2-n} |z-y|^{2-n} \diff z.
\end{align*}
Set $\tilde r=|x-y|$ and let $N\in \N$ be such that $B_{\frac{\tilde r}{N}}(x)\cap B_{\frac{\tilde r}{N}}(y)=\emptyset$. Let $\mathcal{O}=\Omega\setminus (B_{\frac{\tilde r}{N}}(x)\cup B_{\frac{\tilde r}{N}}(y))$ and split the integral over the domain $\Omega$ as the sum of three integrals over the subdomains $B_{\frac{\tilde r}{N}}(x)$, $B_{\frac{\tilde r}{N}}(y)$ and $\mathcal{O}$. Our goal is to estimate $\int_{B_{\frac{\tilde r}{N}}(y)} |z-x|^{2-n} |z-y|^{2-n} \diff z$.

\noindent For $z\in B_{\frac{\tilde r}{N}}(y)$, by the triangular inequality it follows that $|x-z|\geq |x-y|-|y-z|\geq \frac{|x-y|}{\tilde{c}}$ for a suitable constant $\tilde{c}$, then 
\begin{align*}
\int_{B_{\frac{\tilde r}{N}}(y)} |z-x|^{2-n} |z-y|^{2-n} \diff z \leq c |x-y|^{2-n} \int_{B_{\frac{l}{\tilde{4}}}(y)} |z-y|^{2-n} \diff z \leq c |x-y|^{3-n}.
\end{align*}
\noindent Similarly, 
\[
\int_{B_{\frac{\tilde r}{N}}(x)} |z-x|^{2-n} |z-y|^{2-n} \diff z \leq c |x-y|^{4-n}.
\]

\noindent Then, for $z\in \mathcal{O}$, since $|x-z|\geq \frac{|z-y|}{N}$, it follows that
\[
\int_{\mathcal{O}} |z-x|^{2-n} |z-y|^{2-n} \diff z \leq c \int_{\mathcal{O}} |z-y|^{4-2n} \diff z \leq c\int_{\Omega\setminus B_{\frac{\tilde r}{N}}(y)} |z-y|^{4-2n} \diff z \leq c |x-y|^{4-n},
\]
where the constants c appearing in the inequalities depend on the a priori data only. In conclusion, we have proved that
\begin{equation}\label{firstR}
|R(x,y)|\leq C_6 |x-y|^{4-n},
\end{equation}
where $C_6$ depends on the a priori data only.

\noindent The next quantity that we wish to estimate is the gradient of $R$. By a similar argument as in Proposition \ref{prop2}, we pick a cube $Q$ such that $Q\subset B^+_{\frac{r}{4}}$ of side $\frac{cr}{4}$ with $c\in (0,1)$ chosen such that $x\in \p Q$.  By \cite[Lemma 3.2]{g.alessandrini1990}, the following interpolation formula holds
\begin{equation}
\|\nabla R(\cdot,y)\|_{L^{\infty}(Q)} \leq C \|R(\cdot,y)\|_{L^{\infty}(Q)}^{\frac{1}{2}} |\nabla_x R(\cdot,y)|^{\frac{1}{2}}_{1,Q},
\end{equation}
where $C$ depends on $L$ only. Since $G_0$ and $H$ are H\"older continuous, the following estimates hold
\[
|\nabla_x G_0(\cdot,y)|_{1,Q} \leq c |x-y|^{-n} \quad \text{and} \quad |\nabla_x H(\cdot,y)|_{1,Q} \leq c |x-y|^{-n}
\]
where $c$ depends on $L$ only. By \eqref{interpolationA} and\eqref{firstR},
\[
\|\nabla_x R(\cdot,y)\|_{L^{\infty}(Q)} \leq C_7 |x-y|^{2-n},
\]
where $C_7$ depends on the a priori data only.
\end{proof}
\noindent Collecting the results obtained by the two claims, the asymptotic estimates for the Green function follow.
\end{proof}

%

\begin{proof}[Proof of Proposition \ref{prop5}]
The proof follows the lines of \cite[Proposition 3.5]{g.alessandrinim.dicristo2005} and \cite[Theorem 6.5]{g.alessandrinim.dicristoa.morassie.rosset2014}. Let $O\in\p D_1$ be the point of Lemma \ref{metric-lemma}, let $y=h\nu(O)$ where $\nu(O)$ is the outer unit normal of $D_1$ at $O$ . Recall the definition of $S_1$ as 
\begin{equation}\label{S11}
S_1(y,y) = \int_{D_1} (a_{D_1}(x) - a_0(x))A(x)\nabla_x G_1(x,y)\cdot \nabla_x G_2(x,y) \diff x-\int_{D_1} (q_{D_1}(x) - q_0(x))G_1(x,y)\: G_2(x,y) \diff x.
\end{equation}
We can rewrite \eqref{S11} as follows:
\begin{align}\label{splitting}
S_1(y,y) = &\int_{D_1} (a_{D_1}(x) - a_b(x))A(x)\nabla_x H_1(x,y)\cdot\nabla_x H_2(x,y) \diff x +\notag\\
&+ \int_{D_1} (a_{D_1}(x) - a_b(x))A(x)\nabla_x H_1(x,y)\cdot \nabla_x(G_2(x,y)-H_2(x,y)) \diff x +\notag\\
&+ \int_{D_1} (a_{D_1}(x) - a_b(x))A(x)\nabla_x(G_1(x,y)-H_1(x,y))\cdot\nabla_x(G_2(x,y)-H_2(x,y))\diff x + \notag\\
&+ \int_{D_1} (a_{D_1}(x) - a_b(x))A(x)\nabla_x(G_1(x,y)-H_1(x,y))\cdot\nabla_x H_2(x,y) \diff x +\notag \\
&- \int_{D_1} (q_{D_1}(x) - q_b(x)) H_1(x,y)\,H_2(x,y)\diff x -\notag\\
&- \int_{D_1} (q_{D_1}(x) - q_b(x)) H_1(x,y)\,(G_2(x,y)-H_2(x,y)) \diff x -\notag\\
&- \int_{D_1} (q_{D_1}(x) - q_b(x)) (G_1(x,y)-H_1(x,y))\,(G_2(x,y)-H_2(x,y))\diff x -\notag\\
&- \int_{D_1} (q_{D_1}(x) - q_b(x)) (G_1(x,y)-H_1(x,y))\,H_2(x,y).
\end{align}
\noindent Set $\bar{r}_2 = \min\left\{dist(O,D_2),\frac{r_0}{12\sqrt{1+L^2}}\cdot\min\{1,L\} \right\}$. Let  $r\in(0,\bar{r}_2)$. Since for $y=h\nu(O)$, we have that the first term on the righthand side of \eqref{splitting} is the leading term as $h\rightarrow 0^+$, it is convenient to represent the domain of integration as $D_1=(D_1\cap B_r)\cup (D_1\setminus B_r)$.

\noindent so that \eqref{splitting} can be rewritten as follows:
\begin{align}\label{splitting2}
S_1(y,y) = &\int_{D_1\cap B_r(O)} (a_{D_1}(x) - a_b(x))A(x)\nabla_x H_1(x,y)\cdot\nabla_x H_2(x,y) \diff x +\notag\\
&+ \int_{D_1\cap B_r(O)} (a_{D_1}(x) - a_b(x))A(x)\nabla_x H_1(x,y)\cdot \nabla_x(G_2(x,y)-H_2(x,y)) \diff x +\notag\\
&+ \int_{D_1\cap B_r(O)} (a_{D_1}(x) - a_b(x))A(x)\nabla_x(G_1(x,y)-H_1(x,y))\cdot\nabla_x(G_2(x,y)-H_2(x,y))\diff x + \notag\\
&+ \int_{D_1\cap B_r(O)} (a_{D_1}(x) - a_b(x))A(x)\nabla_x(G_1(x,y)-H_1(x,y))\cdot\nabla_x H_2(x,y) \diff x +\notag \\
&+ \int_{D_1\setminus B_r(O)} (a_{D_1}(x) - a_b(x))A(x) \nabla_x G_1(x,y)\cdot \nabla_x G_2(x,y) \diff x -\notag\\
&- \int_{D_1\cap B_r(O)} (q_{D_1}(x) - q_b(x)) H_1(x,y)\,H_2(x,y)\diff x -\notag\\
&- \int_{D_1\cap B_r(O)} (q_{D_1}(x) - q_b(x)) H_1(x,y)\,(G_2(x,y)-H_2(x,y)) \diff x -\notag\\
&- \int_{D_1\cap B_r(O)} (q_{D_1}(x) - q_b(x)) (G_1(x,y)-H_1(x,y))\,(G_2(x,y)-H_2(x,y))\diff x -\notag\\
&- \int_{D_1\cap B_r(O)} (q_{D_1}(x) - q_b(x)) (G_1(x,y)-H_1(x,y))\,H_2(x,y)-\notag\\
&- \int_{D_1\setminus B_r(O)} (q_{D_1}(x) - q_b(x))G_1(x,y)\: G_2(x,y) \diff x.
\end{align}
Set 
\begin{align}
I_1 &= \int_{D_1\cap B_r(O)} (a_{D_1}(x) - a_b(x))A(x)\nabla_x H_1(x,y)\cdot\nabla_x H_2(x,y) \diff x,\\
R_1 &= \int_{D_1\cap B_r(O)} (a_{D_1}(x) - a_b(x))A(x)\nabla_x H_1(x,y)\cdot \nabla_x(G_2(x,y)-H_2(x,y)) \diff x +\notag\\
&+ \int_{D_1\cap B_r(O)} (a_{D_1}(x) - a_b(x))A(x)\nabla_x(G_1(x,y)-H_1(x,y))\cdot\nabla_x(G_2(x,y)-H_2(x,y))\diff x,\\
R_2 &= \int_{D_1\cap B_r(O)} (a_{D_1}(x) - a_b(x))A(x)\nabla_x(G_1(x,y)-H_1(x,y))\cdot\nabla_x H_2(x,y) \diff x,\\
R_3 &= \int_{D_1\setminus B_r(O)} (a_{D_1}(x) - a_b(x))A(x) \nabla_x G_1(x,y)\cdot \nabla_x G_2(x,y) \diff x.
\end{align}
Hence,
\begin{align*}
|S_1(y,y)| \geq |I_1| - |R_1| - |R_2| - |R_3|.
\end{align*}
For the term $I_1$, one can simply notice that
\[
H_1(x,y) = \tilde c\, \Gamma(Jx,Jy),\quad\text{and}\quad H_2(x,y) = \tilde c\,\Gamma(Jx,Jy),
\]
where $J=\sqrt{A(O)^{-1}}$ and $\tilde c$ is a constant that depends only on $a^+, a^-$. Hence, by the uniform ellipticity condition and the lower bound \eqref{lineareta}, 
\[
|I_1| \geq c \int_{D_1\cap B_r(O)} |x-y|^{2-2n} \diff x \geq c r^{2-n}\geq c h^{2-n}
\]
Regarding the term $R_2$, by Proposition \ref{prop2} we know that 
\[
|\nabla_x G_1(x,y) - \nabla_x H_1(x,y)|\leq C |x-y|^{1-n+\theta_2},
\]
so 
\begin{align*}
|R_2| \leq \tilde c \int_{D_1\cap B_r(O)}  |x-y|^{2-2n+\theta_2} \diff x,
\end{align*}
so that
\begin{align*}
|R_2| \leq c h^{2-n+\theta_2}.
\end{align*}
\noindent The term $R_3$ can be bounded in terms of a constant depending on the a priori data only, since $x\neq y$.

\noindent It remains to estimate the term $R_2$. One of the issues is that, by our choice of $r$, there are no asymptotic estimates for the term $\nabla_x(G_2(x,y)-H_2(x,y))$, but we can solve this problem by applying the following trick. Recalling Lemma \ref{lemma2}, one has that $G_2$ has the form
\[
G_2(x,y) = \tilde G_2(x,y) + \sum_{j=1}^{J+1} R_j(x,y),
\]
where $\tilde G_2$ is a weak solution to
\begin{equation}\label{greensystem3}
\left\{ \begin{array}{ll}
\mbox{div}(\sigma_2(\cdot)\nabla \tilde{G}_2(\cdot,y)) = - \delta(\cdot-y) &\textrm{$\textnormal{in}\quad \Omega_0$},\\
\tilde{G}_2(\cdot,y)=0 &\textrm{$\textnormal{on}\quad\p\Omega_0\setminus \Sigma_0$},\\
\sigma_2(\cdot) \nabla \tilde{G}_2(\cdot,y)\cdot \nu(\cdot) + i \tilde{G}_2(\cdot,y)=0 &\textrm{$\textnormal{on}\quad\Sigma_0$}.
\end{array}
\right.
\end{equation}
Hence,
\[
|\nabla_x(G_2(x,y)-H_2(x,y))|\leq |\nabla_x(\tilde G_2(x,y)-H_2(x,y))| + \sum_{j=1}^{J+1} |\nabla_x R_j(x,y)|.
\]
Since 
\[
|\nabla_x R_j(x,y)|\leq c |x-y|^{2j+1-n}, 
\]
for any $j=1,\dots,J-1$, one can infer that
\[
\sum_{j=1}^{J+1} |\nabla_x R_j(x,y)| \leq \sum_{j=1}^{J+1} (d_{\mu}-h)^{2j+1-n} \leq c\,(d_{\mu}-h)^{2-n},
\]
where $d>0$. Regarding the other term, let us first consider a change of variable $\Phi$ as in \cite[Theorem 4.2]{g.alessandrinis.vessella2005} that allows us to flatten the boundary of $\Omega_D$ near the point $O$. Consider $\tilde G_{2,0}(\cdot,y)$ as the Green function which is weak solution to
\begin{equation}\label{greensystem3}
\left\{ \begin{array}{ll}
\mbox{div}(\sigma_{2,0}(\cdot)\nabla \tilde{G}_{2,0}(\cdot,y)) = - \delta(\cdot-y) &\textrm{$\textnormal{in}\quad \Omega_0$},\\
\tilde{G}_{2,0}(\cdot,y)=0 &\textrm{$\textnormal{on}\quad \p \Omega_0\setminus \Sigma_0$},\\
\sigma_{2,0}(\cdot)\nabla \tilde G_{2,0}(\cdot,y)\cdot\nu + i \tilde G_{2,0}(\cdot,y)=0 &\textrm{$\textnormal{on}\quad \Sigma_0,$}
\end{array}
\right.
\end{equation}
where 
\[
\sigma_{2,0}(x)=\left(a_b(0)+(a_{D_2}(0)-a_b(0)\chi_+(x)\right)A(0).
\]
Hence,
\begin{equation}\label{split1}
|\nabla_x(\tilde G_2(x,y)-H_2(x,y))| \leq |\nabla_x(\tilde G_2(x,y) - \tilde G_{2,0}(x,y))| + |\nabla_x(\tilde G_{2,0}(x,y)-H_2(x,y))|
\end{equation}
Regarding the second term on the right-hand side of \eqref{split1}, first notice that $(\tilde G_{2,0}-H_2)(\cdot,y)$ is a weak solution to
\begin{equation}
\begin{cases}
\text{div}_x(\sigma_{2,0}(\cdot)\nabla(\tilde G_{2,0}(\cdot,y)-H_2(\cdot,y))) = 0  &\text{in }B_r(O),\\
\left(\tilde G_{2,0}(\cdot,y)-H_2(\cdot,y)\right)|_{\p B_r(O)} \leq c\,r^{2-n},
\end{cases}
\end{equation}
so that by the Maximum Principle one has that
\[
|\tilde G_{2,0}(x,y)-H_2(x,y)| \leq c r^{2-n}.
\]
Hence, by interior gradient estimates (see for instance \cite{d.gilbargn.s.trudinger2001}), it follows that
\begin{equation}\label{Gest}
|\nabla_x(\tilde G_{2,0}(x,y)-H_2(x,y))| \leq c\,r^{1-n}.
\end{equation}
For the first term on the right-hand side of \eqref{split1}, define
\[
\tilde R_2(x,y) = \tilde G_2(x,y) - \tilde G_{2,0}(x,y).
\]
One can notice that $\tilde R_2(\cdot,y)$ is a weak solution to
\begin{equation*}
\begin{cases}
\text{div}(\sigma_2(\cdot)\nabla \tilde R_2(\cdot,y)) = - \text{div}\big((\sigma_2(\cdot)-\sigma_{2,0}(\cdot))\nabla \tilde G_{2,0}(\cdot,y)\big) &\text{in }\Omega_0\\
\tilde R_2(\cdot,y) = 0 &\text{on }\p \Omega_0\setminus\Sigma_0,\\
\sigma_2(\cdot)\nabla \tilde R_2(\cdot,y)\cdot\nu + i\tilde R_2(\cdot,y)=-(\sigma_{2}(\cdot)-\sigma_{2,0}(\cdot))\nabla \tilde G_{2,0}(\cdot,y)\cdot\nu &\text{on }\Sigma_0.
\end{cases}
\end{equation*}
By the representation formula, the remainder has the form
\begin{align}
-\tilde R_2(x,y) = &\int_{\Omega_0} (\sigma_2(z)-\sigma_{2,0}(z))\nabla_z \tilde G_{2}(z,x)\cdot \nabla_z \tilde G_{2,0}(z,y) \diff z +\notag\\
&+ \int_{\p \Omega_0} \sigma_{2,0}(z)\nabla_z \tilde G_{2,0}(z,y)\cdot \nu \left[ \tilde G_{2,0}(z,x) -\tilde G_{2,0}(z,x)\right] \diff S(z) +\notag\\
&+ \int_{\p\Omega_0} \sigma_2(z)\nabla_z \left[\tilde G_2(z,x) - \tilde G_{2,0}(z,x)\right]\cdot\nu\,\, \tilde G_{2,0}(z,y) \diff S(z).
\end{align}
The integral over $\p\Omega_0$ are bounded from above by a positive constant that depends on the a priori data only. In order to estimate the volume integral, first notice that
\[
|\sigma_2(z)-\sigma_{2,0}(z)|\leq C |z|,
\]
where $C$ is a positive constant depending only on a priori data.

\noindent Hence, by Proposition \ref{prop1},
\begin{align}
\Big| \int_{\Omega_0} &(\sigma_2(z)-\sigma_{2,0}(z))\nabla_z \tilde G_2(z,x)\cdot \nabla_z \tilde G_{2,0}(z,y) \diff z \Big| \leq\notag\\
&\leq c\int_{\Omega_0} |z|\,|z-x|^{1-n}\,|z-y|^{1-n} \diff z,
\end{align}
where $c$ is a positive constant depending on the a priori data only. 
Set  $\tilde h= |x-y|$ and define
%
%
\begin{align}
I_1 &= \int_{B_{4\tilde h}} |z|\,|z-x|^{1-n}\,|z-y|^{1-n} \diff z,\\
I_2 &= \int_{\R^n\setminus B_{4\tilde h}} |z|\,|z-x|^{1-n}\,|z-y|^{1-n} \diff z.
\end{align}
so that
\begin{equation}
|\tilde R_2(x,y)| \leq c(I_1 + I_2).
\end{equation}
First, let us estimate $I_1$. Set $z=\tilde h w$, $t=\frac{x}{\tilde h}$ and $s=\frac{y}{\tilde h}$, then
\begin{align*}
I_1 &= \int_{B_{4}} \tilde h |w|\,|\tilde h(w-t)|^{1-n}\,|\tilde h(w-s)|^{1-n} \tilde h \diff w\\
&= 4\tilde h^{3-n} \int_{B_4} |w-t|^{1-n}\,|w-s|^{1-n}\,\diff w\\
&\leq c \tilde h^{3-n},
\end{align*}
since $\int_{B_4} |w-t|^{1-n}\,|w-s|^{1-n}\,\diff w \leq c$ (see \cite[Chapter 2, section 11]{c.miranda1970}). Hence,
\begin{equation}\label{I1}
I_1 \leq c (h-dist(O,D_2))^{3-n}.
\end{equation}
Regarding the integral $I_2$, notice that since $y=h\nu(O)=-he_n$ in a suitable coordinate system, we might choose $h$ so that
\[
|y| = -h \leq |x-y| = \tilde h
\]
and
\[
|x|\leq |x-y|+|y| \leq 2\tilde h.
\]
For any $z\in \R^n\setminus B_{4\tilde h}$, since $|z|>4\tilde h$, it follows that
\begin{equation*}
\frac{3}{4}|z| \leq |z-y|\quad \text{and}\quad \frac{1}{2}|z| \leq |z-x|.
\end{equation*}
It follows that
\begin{equation}\label{I2}
I_2 \leq \left(\frac{8}{3}\right)^{1-n} \int_{\R^n\setminus B_{4\tilde h}} |z|^{3-2n} \diff z \leq c\tilde h^{3-n} \leq c (h-dist(O,D_2))^{3-n}.
\end{equation}
By \eqref{I1} and \eqref{I2}, we can conclude that
\begin{equation}\label{R2}
|\tilde R_2(x,y)| \leq c |x-y|^{3-n}.
\end{equation}
At this point, in order to determine an upper bound for $\nabla_x \tilde R_2$. Consider a cube $Q\subset D_1\cap B_r(O)$. Since $\tilde G_{2}(\cdot,y)$ and $\tilde G_{2,0}(\cdot,y)$ are H\"older continuous, it follows that
\[
|\nabla \tilde R_2(x,y)|_{\alpha,Q} \leq c |x-y|^{-n}.
\]
By the known inequality,
\[
\|\nabla \tilde R_2(\cdot,y)\|_{L^{\infty}(Q)} \leq \|\tilde R_2(\cdot,y)\|^{\frac{1}{2}}_{L^{\infty}(Q)} |\nabla \tilde R_2(\cdot,y)|^{\frac{1}{2}}_{1,Q},
\]
by \eqref{R2} it follows that
\begin{equation}\label{GR2}
|\nabla R(x,y)| \leq c |x-y|^{1-n+\theta_3}, \quad \text{where }\quad\theta_3 = \frac{1}{2}.
\end{equation}
Collecting \eqref{split1}, \eqref{Gest} and \eqref{GR2} together, we obtain
\begin{equation}\label{estGfinal}
|\nabla_x(\tilde G_2(x,y) - H_2(x,y))| \leq c h^{1-n+\theta_3}.
\end{equation}
\noindent In conclusion, the lower bound of $S_1$ is given by
\begin{equation*}
|S_1(y,y)| \geq c h^{2-n}.
\end{equation*}
Regarding the estimate for $S_2$, from Proposition \ref{prop1} it follows that
\begin{align*}
|S_2(y,y)| &\leq C \int_{D_2} |x-y|^{1-n}|x-y|^{1-n} \diff x \leq C h^{2(1-n)}.
\end{align*}
In conclusion,
\begin{equation}
|f(y,y)|=|S_1(y,y)-S_2(y,y)|\geq |S_1(y,y)| - |S_2(y,y)| \geq c_2 h^{2-n}-c_3 h^{2(1-n)},
\end{equation}
for suitable $c_1>0$ constant depending on the a priori data only.

\end{proof}


\section{The misfit functional}\label{5}
In this section we introduce a stability estimate in terms of the misfit functional that is defined in \eqref{misfit}.

\noindent Let $\Omega$, $D_1$, $D_2$ be, respectively, a bounded domain satisfying \eqref{omegabegin}-\eqref{omegaend} and two inclusions satisfying \eqref{inclusionbegin}-\eqref{inclusionend}. Let $\sigma_1, \sigma_2, q_1. q_2$ be the jump coefficients that correspond to the two inclusions. Let $G_i$ be the Green functions associated to the operator $\mbox{div}(\sigma_{i}(\cdot)\nabla\cdot)+q_i$ for $i=1,2$ so that for $y\in D_0$, $G_i(\cdot,y)$ is a distributional solution to the boundary value problem \eqref{greensystem}. Pick $D_y, D_z\subset \subset D_0$ suitable Lipschitz domains whose intersection is empty. For $(y,z)\in D_0\times D_0$, define
\begin{equation}\label{su0}
S_{\mathcal{U}_0}(y,z) = \int_{\Sigma} \left[ \sigma_{1}(x)\nabla G_1(x,y)\cdot\nu(x)\: G_2(x,z) - \sigma_{2}(x)\nabla G_2(x,z)\cdot\nu(x)\: G_1(x,y)\right]\:dS(x),
\end{equation}
where $\Sigma$ is the open portion of the boundary of $\Omega$ where the measurements are performed.

\noindent The \textit{misfit functional} is defined as
\begin{equation}\label{misfit}
\mathcal{J}(D_1,D_2)=\int_{D_y\times D_z} \left|S_{\mathcal{U}_0}(y,z)\right|^2dy\, dz.
\end{equation}
where $\mathcal{J}:L^{\infty}(\Omega_0)\times L^{\infty}(\Omega_0)\rightarrow \R$ is a functional and encodes the error that occur when one approximates the boundary data induced by $\sigma_1$ and $q_1$ by the one induced by $\sigma_2$ and $q_2$.

\subsection{\normalsize The Stability estimate}	
Let  $\omega:[0,+\infty)\rightarrow [0,+\infty)$ be an non-decreasing function such that for any $t\in (0,1)$, $\omega(t)\leq C\cdot |\ln t|^{-\eta}$, where $\eta\in (0,1)$ is a suitable constant.
\begin{theorem}\label{theoremisfit}

Let $\Omega\subset\R^n$ be a bounded domain satisfying \eqref{omegabegin}-\eqref{omegaend} and let $D_1$, $D_2$ be two inclusions of $C^{2}$ class contained in $\Omega$ satisfying \eqref{inclusionbegin}-\eqref{inclusionend}. Let $\sigma_{1}$ and $\sigma_{2}$ be the anisotropic conductivities satisfying \eqref{condbegin}-\eqref{condend} and let $q_1$ and $q_2$ be the coefficients of the zero order term satisfying \eqref{zerorderbegin}-\eqref{zerorderend}. Let $\Sigma$ be a non-empty open portion of $\p\Omega$. For any $\epsilon\in (0,1)$, if $\mathcal{J}(\sigma_{1},\sigma_{2})<\epsilon$, then 
\begin{equation}
d_H(\p D_1, \p D_2)\leq \omega(\epsilon),
\end{equation}
where $C>0$ is a constant that depends on the a priori data only.
\end{theorem}

\noindent The proof of Theorem \ref{theoremisfit} follows the lines of of the proof of Theorem \ref{theorem}, but instead of Proposition \ref{prop5} we need to introduce Proposition \ref{prop4}. Before stating it, notice that by Green's identity and \eqref{green} we have that  \eqref{su0} that can be rewritten as
\begin{equation}\label{su0new}
S_{\mathcal{U}_0}(y,z) = \int_{\Omega} \big(\sigma_{2}(x)-\sigma_{1}(x)\big) \nabla G_1(x,y)\cdot\nabla G_2(x,z) + \int_{\Omega} \big(q_{1}(x)- q_{2}(x)\big) G_1(x,y) G_2(x,z) \diff x.
\end{equation}
Hence, by the definition of \eqref{S1}, \eqref{S2} and \eqref{f}, we have that 
\[
S_{\mathcal{U}_0}(y,z) = f(y,z).
\]

\begin{proposition}\label{prop4}
Under the same assumptions of Theorem \ref{theoremisfit}, for $\epsilon\in(0,1)$, if $\mathcal{J}(\sigma_{D_1},\sigma_{D_2})<\epsilon$, then 
\[
|f(y,y)| \leq C_1 \frac{\epsilon^{Bh^F}}{h^A}, 
\]
where $A\in (0,1)$, $C_1, B, F$ are positive constants that depend on the a priori data only, $y=h\nu(O)$ where
\[
0<h\leq \bar{d}\left(1-\frac{\sin\theta_0}{4}\right) \mbox{ for }\,\theta_0=\arctan\frac{1}{L},
\]
.
\end{proposition}

\begin{proof}[Proof of Proposition \ref{prop4}]
As in the previous proofs, we drop the indices and consider a generic inclusion $D$. Fix $\bar{y}\in D_0$, then $f(\bar{y},\cdot)$ is a weak solution to
\[
\text{div}_z(\sigma(\cdot)\nabla_z f(\bar{y},\cdot)) + q(\cdot) f(\bar{y},\cdot) = 0, \qquad \text{in }\Omega^c_D.
\]
Since in this case $f(\bar{y},z)=S_{\mathcal{U}_0}(\bar{y},z)$, by \cite[(3.23)]{s.foschiattie.sincichr.gaburro2021} we have
\[
\max_{z\in (D_0)_r} S_{\mathcal{U}_0}(\bar{y},z) \leq c \mathcal{I}(D_1, D_2)
\]
where $c$ depends on the a priori data only. Then
\begin{equation}
f(\bar{y},z)\leq c\varepsilon.
\end{equation}
The remaining part of the proof follows the line of the proof of Proposition \ref{prop-bis}.
\end{proof}


\section*{Acknowledgments}
The work of SF and ES was performed under the PRIN grant No. 201758MTR2-007. ES has also been supported by Gruppo Nazionale per l'Analisi Matematica, la Probabilità e le loro applicazioni (GNAMPA) by the grant "Problemi inversi per equazioni alle derivate parziali".


\end{document}